\documentclass[a4,11pt,fleqn]{article}
\usepackage{pstricks}
\usepackage{enumitem}
\usepackage{amsmath}
\usepackage{amssymb}
\usepackage{pifont}
\usepackage{theorem} 
\usepackage{euscript}
\usepackage{exscale,relsize}
\usepackage{epic,eepic}
\usepackage{multicol}
\usepackage{makeidx}
\usepackage{srcltx}         
\usepackage{xcolor}
\usepackage{url}
\usepackage{charter}
\usepackage{mathrsfs} 
\newcommand{\email}[1]{\href{mailto:#1}{\nolinkurl{#1}}}
\usepackage[normalem]{ulem}     
\newlength{\mySubFigSize}
\definecolor{labelkey}{rgb}{0,0.08,0.45}
\definecolor{refkey}{rgb}{0,0.6,0.0}
\definecolor{Brown}{rgb}{0.45,0.0,0.05}
\definecolor{dgreen}{rgb}{0.00,0.49,0.00}
\definecolor{dblue}{rgb}{0,0.08,0.75}
\RequirePackage[dvips,colorlinks,hyperindex]{hyperref}
\hypersetup{linktocpage=true,citecolor=dblue,linkcolor=dgreen}
\renewcommand{\leq}{\ensuremath{\leqslant}}
\renewcommand{\geq}{\ensuremath{\geqslant}}

\newcommand{\minimize}[2]{\ensuremath{\underset{\substack{{#1}}}%
{\text{minimize}}\;\;#2 }}

\newcommand{\Frac}[2]{\displaystyle{\frac{#1}{#2}}} 
\newcommand{\Scal}[2]{\bigg\langle{#1}\;\bigg|\:{#2}\bigg\rangle} 
\newcommand{\scal}[2]{{\langle{{#1}\mid{#2}}\rangle}}

\newcommand{\menge}[2]{\big\{{#1}~\big |~{#2}\big\}} 
\newcommand{\Menge}[2]{\left\{{#1}~\left|~{#2}\right.\right\}} 
\newcommand{\KKK}{\ensuremath{\boldsymbol{\mathcal K}}}

\newcommand{\HH}{\ensuremath{{\mathcal H}}}
\newcommand{\GG}{\ensuremath{{\mathcal G}}}

\newcommand{\Sum}{\ensuremath{\displaystyle\sum}}
\newcommand{\emp}{\ensuremath{{\varnothing}}}

\newcommand{\Id}{\ensuremath{\operatorname{Id}}\,}
\newcommand{\cart}{\ensuremath{\raisebox{-0.5mm}{\mbox{\LARGE{$\times$}}}}}

\newcommand{\RR}{\ensuremath{\mathbb{R}}}

\newcommand{\RPP}{\ensuremath{\left]0,+\infty\right[}}

\newcommand{\RX}{\ensuremath{\left]-\infty,+\infty\right]}}

\newcommand{\NN}{\ensuremath{\mathbb N}}

\newcommand{\weakly}{\ensuremath{\:\rightharpoonup\:}}

\newcommand{\ran}{\ensuremath{\text{\rm range}\,}}

\newcommand{\pinf}{\ensuremath{{+\infty}}}

\newcommand{\prox}{\ensuremath{\text{\rm prox}}}

\newcommand{\gra}{\ensuremath{\text{\rm gra}\,}}

\newcommand{\zeroun}{\ensuremath{\left]0,1\right[}}

\newcommand{\braces}[1]{\left\{#1\right\}}

\newcommand{\Bigbraces}[1]{\Big\{{#1}\Big\}}

\newtheorem{theorem}{Theorem}[section]
\newtheorem{lemma}[theorem]{Lemma}

\newtheorem{proposition}[theorem]{Proposition}
\newtheorem{assumption}[theorem]{Assumption}
\theoremstyle{plain}{\theorembodyfont{\rmfamily}%
}
\theoremstyle{plain}{\theorembodyfont{\rmfamily}%
}
\theoremstyle{plain}{\theorembodyfont{\rmfamily}%
\newtheorem{remark}[theorem]{Remark}}
\theoremstyle{plain}{\theorembodyfont{\rmfamily}%
\newtheorem{algorithm}[theorem]{Algorithm}}
\theoremstyle{plain}{\theorembodyfont{\rmfamily}%
}
\theoremstyle{plain}{\theorembodyfont{\rmfamily}%
}
\theoremstyle{plain}{\theorembodyfont{\rmfamily}%
}
\theoremstyle{plain}{\theorembodyfont{\rmfamily}%
\newtheorem{problem}[theorem]{Problem}}

\numberwithin{equation}{section}
\begin{document}

\title{\sffamily\LARGE Asynchronous Block-Iterative Primal-Dual
Decomposition Methods for Monotone 
Inclusions\footnote{Contact author: P. L. Combettes, 
\email{plc@ljll.math.upmc.fr},
phone: +33 1 4427 6319, fax: +33 1 4427 7200.
The work of P. L. Combettes was supported by the 
CNRS MASTODONS project under grant 2013MesureHD and by the 
CNRS Imag'in project under grant 2015OPTIMISME.}}

\author{
Patrick L. Combettes$\,^1$~ and ~Jonathan Eckstein$\,^2$
\\[5mm]
\small $^1$Sorbonne Universit\'es -- UPMC Univ. Paris 06\\
\small UMR 7598, Laboratoire Jacques-Louis Lions\\
\small F-75005, Paris, France\\
\small \email{plc@ljll.math.upmc.fr}\\[5mm]
\small $^2$Department of Managemement Science and Information 
Systems and RUTCOR\\
\small Rutgers University\\
\small Piscataway, NJ 08854, USA\\
\small \email{jeckstei@rci.rutgers.edu}\\[4mm]
}

\date{~}

\maketitle

\vskip 8mm

\begin{abstract} 
\noindent
We propose new primal-dual decomposition algorithms for solving 
systems of inclusions involving sums of linearly composed maximally 
monotone operators. The principal innovation in these algorithms is 
that they are block-iterative in the sense that, at each iteration, 
only a subset of the monotone operators needs to be processed, as
opposed to all operators as in established methods. 
Deterministic strategies are used to select the blocks 
of operators activated at each 
iteration. In addition, we allow for operator processing ``lags'', 
permitting asynchronous implementation.
The decomposition phase of each iteration of our methods 
is to generate points in the graphs of 
the selected monotone operators, in order to construct a half-space 
containing the Kuhn-Tucker set associated with the system. 
The coordination phase of each iteration involves a projection 
onto this half-space.
We present two related methods: 
the first method provides weakly convergent primal and dual
sequences under general conditions, while the second is a variant 
in which strong convergence is guaranteed without additional
assumptions. Neither algorithm requires prior knowledge of bounds
on the linear operators involved or the inversion of linear
operators. Our algorithmic framework unifies and significantly
extends the approaches taken in earlier work on primal-dual
projective splitting methods. 
\end{abstract} 


\maketitle

\newpage

\section{Introduction}
This paper considers systems of monotone inclusions of the 
following general form.

\begin{problem}
\label{prob:2}
Let $m$ and $p$ be strictly positive integers, set $I=\{1,\ldots,m\}$
and $K=\{1,\ldots,p\}$, and let $(\HH_i)_{i\in I}$ and 
$(\GG_k)_{k\in K}$ be real Hilbert spaces. For every $i\in I$ and 
$k\in K$, let $A_i\colon\HH_i\to 2^{\HH_i}$ and 
$B_k\colon\GG_k\to 2^{\GG_k}$ 
be maximally monotone, let $z^*_i\in\HH_i$, let $r_k\in\GG_k$, 
and let $L_{ki}\colon\HH_i\to\GG_k$ be linear and bounded. 
Consider the coupled inclusions problem
\begin{equation}
\label{eferw7er3h02-26p}
\text{find}\;\;(\overline{x}_i)_{i\in I}\in\bigoplus_{i\in I}\HH_i
\;\;\text{such that}\;\;(\forall i\in I)\quad
z^*_i\in A_i\overline{x}_i+\Sum_{k\in K}L_{ki}^*
\bigg(B_k\bigg(\Sum_{j\in I}L_{kj}\overline{x}_j-r_k\bigg)\bigg),
\end{equation}
its dual problem 
\begin{multline}
\label{eferw7er3h02-26d}
\text{find}\;\;(\overline{v}^*_k)_{k\in K}\in\bigoplus_{k\in K}
\GG_k\;\;\text{such that}\\ 
(\forall k\in K)\quad -r_k\in-\Sum_{i\in I}L_{ki}\bigg(A_i^{-1}
\bigg(z^*_i-\Sum_{l\in K}L_{li}^*\overline{v}^*_l\bigg)\bigg)
+B_k^{-1}\overline{v}^*_k,
\end{multline}
and the associated Kuhn-Tucker set
\begin{multline}
\label{eferw7er3h02-26k}
\boldsymbol{Z}=\bigg\{\big((\overline{x}_i)_{i\in I},
(\overline{v}^*_k)_{k\in K}\big)\;\bigg |\;
(\forall i\in I)\;\;\overline{x}_i\in\HH_i\;\;\text{and}\;\;
z^*_i-\sum_{k\in K}L_{ki}^*\overline{v}_k^*\in
A_i\overline{x}_i,\:\;\text{and}\\
(\forall k\in K)\;\;\overline{v}_k^*\in\GG_k\;\;\text{and}\;\;
\sum_{i\in I}L_{ki}\overline{x}_i-r_k\in B_k^{-1}\overline{v}_k^*
\bigg\}.
\end{multline}
The problem is to find a point in $\boldsymbol{Z}$.
The sets of solutions to \eqref{eferw7er3h02-26p} and 
\eqref{eferw7er3h02-26d} are denoted by $\mathscr{P}$ and 
$\mathscr{D}$, respectively.
\end{problem}

As discussed in \cite{Siop14}, Problem~\ref{prob:2} models a wide
range of problems arising game theory, image recovery, evolution 
equations, machine learning, signal processing, mechanics, 
the cognitive sciences, and domain decomposition methods in 
partial differential equations. In \cite[Section~5]{Siop13}, it 
was shown that an important special case of Problem~\ref{prob:2} 
is the following optimization problem, in which the monotone 
operators $(A_i)_{i\in I}$ and $(B_k)_{k\in K}$ are taken to 
be subdifferentials.

\begin{problem}
\label{prob:3}
Let $m$ and $p$ be strictly positive integers, set $I=\{1,\ldots,m\}$
and $K=\{1,\ldots,p\}$, and let $(\HH_i)_{i\in I}$ and 
$(\GG_k)_{k\in K}$ be real Hilbert spaces. For every $i\in I$ and 
$k\in K$, let $f_i\colon\HH_i\to\RX$ and
$g_k\colon\GG_k\to\RX$ be proper lower semicontinuous convex
functions, let $z^*_i\in\HH_i$, let $r_k\in\GG_k$, 
and let $L_{ki}\colon\HH_i\to\GG_k$ be linear and bounded. 
Suppose that
\begin{equation}
\label{e:2012-10-21a}
(\forall i\in I)\quad
z^*_i\in\ran\bigg(\partial f_i+\sum_{k\in K}L_{ki}^*\circ
\partial g_k\circ\bigg(\sum_{j\in I}L_{kj}\cdot-r_k\bigg)\bigg).
\end{equation}
The problem is to solve the primal minimization problem
\begin{equation}
\label{e:2012-10-23p}
\minimize{(x_i)_{i\in I}\in\bigoplus_{i\in I}\HH_i}
{\sum_{i\in I}\big(f_i(x_i)-\scal{x_i}{z^*_i}\big)+\sum_{k\in K} 
g_k\bigg(\sum_{i\in I}L_{ki}x_i-r_k\bigg)}
\end{equation}
along with its dual problem
\begin{equation}
\label{e:2012-10-23d}
\minimize{(v^*_k)_{k\in K}\in\bigoplus_{k\in K}\GG_k}
{\sum_{i\in I}f_i^*\bigg(z^*_i-\sum_{k\in K}L_{ki}^*v^*_k\bigg)
+\sum_{k\in K}\big(g^*_k(v^*_k)+\scal{v^*_k}{r_k}\big)}.
\end{equation}
\end{problem}

In recent years, several decomposition algorithms have been 
proposed to solve Problem~\ref{prob:2} (or at least the primal 
problem \eqref{eferw7er3h02-26p}) under various hypotheses 
\cite{Siop14,Nfao15,Sico10,Bot15b,Bot13b,Bric15,%
Siop13,Siop15,Davi15,Pesq15}. 
In such algorithms, the monotone operators as well as the 
linear operators are evaluated individually.  
The methods we propose in the present paper for solving 
Problem~\ref{prob:2} are based on those of 
\cite{Siop14,Nfao15}, which are themselves based on the 
projective primal-dual methods initiated in \cite{Svai08,Svai09} 
for finding a zero of the sum of monotone operators. The basic idea
underlying this class of methods is to generate at each iteration
points in the graphs of all the monotone operators in such a way 
as to construct a half-space containing the Kuhn-Tucker set 
$\boldsymbol{Z}$. The calculations of each of
these points are resolvent computations involving a single
monotone operator $A_i$ or $B_k$, 
which is what makes the methods splitting algorithms.
The coordination step of the method is to project the current
iterate onto the recently constructed half-space.
The advantages of this approach are that it does not impose 
additional assumptions on the operators present in the 
formulation, it does not require knowledge of the norm of 
the linear operators $(L_{ik})_{i\in I,k\in K}$ or of 
combinations thereof, and it does not involve the inversion of 
linear operators. 

The methods of \cite{Siop14,Nfao15} must evaluate
all $m+p$ resolvents of the operators $(A_i)_{i\in I}$ and 
$(B_k)_{k\in K}$ at every iteration, with only limited
ability to pass information between these calculations.
Essentially, the resolvents of all the operators $(A_i)_{i\in I}$ 
must be evaluated independently, and then similarly for all 
the operators $(B_k)_{k\in K}$. In this setting, the only
information flow within each iteration is from the 
$(A_i)_{i\in I}$ calculations to the 
$(B_k)_{k\in K}$ calculations. This property results in 
an algorithm in which large blocks
of calculations must be performed before any information is
exchanged between subsystems. Although in principle conducive to
parallel computing, this kind of structure can still lead to
difficulties even in a parallel execution environment: it requires
an essentially synchronous implementation, so if some small
subset of the subsystems represented by the operators
$(A_i)_{i\in I}$ or $(B_k)_{k\in K}$ are more
computation-intensive than others, load balancing can become
problematic: most processors may have to sit idle while the
remaining few complete their tasks.  This kind of structure is
common to nearly all prior splitting schemes for more than two
monotone operators, the only exception we are aware of
being that of \cite{Svai09} for the special case
\begin{equation}
\label{e:2009}
\text{find}\;\;\overline{x}\in\HH\;\;\text{such that}\;\; 
0\in\sum_{k\in K}B_k\overline{x}
\end{equation}
of \eqref{eferw7er3h02-26p}.
In that case, information can flow in fairly
arbitrary ways between the $p$ resolvent calculations comprising
each iteration, as described by a set of algorithm parameters that
is quadratic in $p$; the selection of these parameters is subject 
to a specific eigenvalue condition. However, the algorithm is still
fundamentally synchronous, and it is has never been clear how to
select its many parameters.

This paper presents a different approach to constructing more
flexible and potentially asynchronous decomposition methods for
problems fitting the general structure represented
by Problem~\ref{prob:2}. The key idea is that our algorithm has the
ability to process an essentially arbitrary subset of the
operators between successive coordination/projection operations.
The only restriction is one adapted from block-iterative methods
for convex feasibility problems \cite{Baus96,Jamo97,Otta88}: 
for some possibly large positive integer
$M$, each operator must be processed at least once over every span
of $M$ consecutive iterations. 
To our knowledge, this is the first application
of this kind of versatile deterministic 
control scheme to finding zeros of sums of
operators. Such control schemes have been used in 
convex feasibility problems \cite{Baus96,Jamo97}. 
This aspect of our algorithm gives it 
potential flexibility absent from other splitting schemes
for monotone inclusions: first, it provides the ability to find an
arbitrary balance between computational effort expended on the
subsystems and that expended on coordination.  For example, if the
subsystems are relatively time-consuming to process, one could
perform as few as a single subsystem evaluation between successive
projection steps, with the projections immediately spreading
the information from each subsystem evaluation to each successive
one. The second aspect of the flexibility of our approach 
involves the balance of computational effort between subsystems: in
prior decomposition methods for monotone inclusions, every operator
must be processed exactly the same number of times, but the class
of algorithms proposed here is much more flexible. 
If, for example, some operators are less time-consuming to process
than others, one has the option of processing them more frequently. 
Such features can be very useful in applications such as 
those described in \cite{Nume15}.

Our analysis allows each activation of an operator to use
information originating from an earlier iteration than the one
in which its results are incorporated into the computation. This
feature makes it possible to implement the algorithm 
asynchronously: the points in the graphs of the monotone operators
incorporated into the projection step during a given iteration 
may be the results of resolvent computations initiated during 
earlier iterations. 
Our analysis shows that our method still converges 
so long as there is a fixed (but arbitrary) upper bound on 
the number of 
iterations between initiation and incorporation of a 
resolvent calculation. The potentially asynchronous nature of 
our method is a significant asset in the design of efficient 
parallel implementations. 

Prior work on projective splitting methods has used two different
approaches to constructing affine half-spaces to separate the 
target set $\boldsymbol{Z}$ from the current iterate. The original 
approach in \cite{Svai08,Svai09} was developed for the inclusion
problem \eqref{e:2009}. In this special case
of \eqref{eferw7er3h02-26p}, it was
possible to efficiently confine the iterates to a specific 
subspace $\KKK$ of the primal-dual space, which can be numerically 
advantageous. 
In the general setting of Problem~\ref{prob:2}, the analysis of
\cite{Siop14,Nfao15} used an alternative 
half-space construction in 
which the iterates are not confined to a subspace. A secondary
contribution of this paper is to develop a unifying framework
for constructing separators for $\boldsymbol{Z}$ in which both
prior approaches appear as special cases.

We present two classes of algorithms based on many of the same
underlying building blocks and which may be viewed as asynchronous
block-iterative extensions of the algorithms of
\cite{Siop14,Nfao15}. The first class uses a straightforward
half-space projection at each iteration and allows for conventional
overrelaxation of the projection steps by factors upper bounded by
$2$. This class exhibits weak convergence to an unspecified
Kuhn-Tucker point. The second class is a variant that involves a
more complicated projection operation and does not use
overrelaxation, but induces strong convergence to the unique point
in the Kuhn-Tucker set that best approximates a given reference
point. Numerical experiments with these new algorithms are being 
conducted and we shall report on their results elsewhere.

When applied in suitable product spaces, the block-coordinate
methods of \cite{Siop15,Pesq15} can be used to derive
block-iterative splitting algorithms methods for a certain class of
problems. However, unlike the
methods we propose here,
the resulting algorithms have been proved
to converge only under random operator selection strategies, and
they require either joint cocoercivity assumptions on the
operators $(B_k)_{k\in K}$ or the ability to block-decompose 
the projection onto the graph of certain linear operators. 

{\bfseries Notation.} 
Our notation is standard and follows \cite{Livre1}, which 
contains the
necessary background on monotone operators and convex analysis. 
The scalar product of a Hilbert space is denoted by 
$\scal{\cdot}{\cdot}$ and the associated norm by $\|\cdot\|$.
The projection operator onto a nonempty closed convex subset $C$ 
of $\HH$ is denoted by $P_C$. The symbols $\weakly$ and $\to$ 
denote respectively weak and strong convergence, and $\Id$ 
denotes the identity operator. The Hilbert direct 
sum of two Hilbert spaces $\HH$ and $\GG$ is denoted by 
$\HH\oplus\GG$, and the power set of $\HH$ by
$2^{\HH}$. Given $A\colon\HH\to 2^{\HH}$, 
$\gra A$ denotes the graph of $A$, 
$A^{-1}$ denotes the inverse of $A$, 
and $J_A=(\Id+A)^{-1}$ denotes the resolvent of $A$. 

\section{Analysis of a generic primal-dual composite inclusion 
problem}
\label{sec:2}

\subsection{Problem statement}

Our investigation will be simplified by the analysis of the 
following problem, which can be regarded as a reduction of
Problem~\ref{prob:2} to the case when $m=K=1$, $z^*_1=0$, 
and $r_1=0$.

\begin{problem}
\label{prob:1}
Let $\HH$ and $\GG$ be real Hilbert spaces. Let 
$A\colon\HH\to 2^{\HH}$ and $B\colon\GG\to 2^{\GG}$ be maximally 
monotone operators, and let $L\colon\HH\to\GG$ be a bounded linear 
operator. Consider the inclusion problem
\begin{equation} 
\label{e:primal}
\text{find}\;\;\overline{x}\in\HH\;\;\text{such that}\;\; 
0\in A\overline{x}+L^*BL\overline{x},
\end{equation}
its dual problem
\begin{equation}
\label{e:dual}
\text{find}\;\;\overline{v}^*\in\GG\;\;\text{such that}\;\; 
0\in -LA^{-1}(-L^*\overline{v}^*)+B^{-1}\overline{v}^*,
\end{equation}
and the associated Kuhn-Tucker set
\begin{equation}
\label{e:2013-05-21a}
\boldsymbol{Z}=\menge{(x,v^*)\in\HH\oplus\GG}
{-L^*v^*\in Ax\:\;\text{and}\;Lx\in B^{-1}v^*}.
\end{equation}
The problem is to find a point in $\boldsymbol{Z}$.
The sets of solutions to \eqref{e:primal} and \eqref{e:dual} are 
denoted by $\mathscr{P}$ and $\mathscr{D}$, respectively.
\end{problem}

\begin{proposition}
\label{pgre76g12-12}
Consider the setting of Problem~\ref{prob:1} and let $\KKK$ 
be a closed vector subspace of $\HH\oplus\GG$ such that
$\boldsymbol{Z}\subset\KKK$. Then the following hold:
\begin{enumerate}
\item
\label{pgre76g12-12i}
$\boldsymbol{Z}$ is a closed convex subset of 
$\mathscr{P}\times\mathscr{D}$.
\item
\label{pgre76g12-12ii}
$\mathscr{P}\neq\emp\Leftrightarrow\boldsymbol{Z}
\neq\emp\Leftrightarrow\mathscr{D}\neq\emp$.
\item
\label{pgre76g12-12iii}
For every $\mathsf{a}=(a,a^*)\in\gra A$ and 
$\mathsf{b}=(b,b^*)\in\gra B$, set
$\boldsymbol{s}_{\mathsf{a},\mathsf{b}}^*=(a^*+L^*b^*,b-La)$, 
$\boldsymbol{t}_{\mathsf{a},\mathsf{b}}^*=P_{\KKK}
\boldsymbol{s}_{\mathsf{a},\mathsf{b}}^*\,$,
$\eta_{\mathsf{a},\mathsf{b}}=\scal{a}{a^*}+\scal{b}{b^*}$, and 
\begin{equation}
\boldsymbol{H}_{\mathsf{a},\mathsf{b}}=
\menge{\boldsymbol{x}\in\boldsymbol{\mathcal{K}}}
{\scal{\boldsymbol{x}}{\boldsymbol{t}_{\mathsf{a},\mathsf{b}}^*}
\leq\eta_{\mathsf{a},\mathsf{b}}}.
\end{equation}
Then the following hold:
\begin{enumerate}
\item
\label{pgre76g12-12iiia}
Let $\mathsf{a}\in\gra A$ and $\mathsf{b}\in\gra B$. Then
$\boldsymbol{H}_{\mathsf{a},\mathsf{b}}=\KKK
\;\Leftarrow\;\boldsymbol{s}^*_{\mathsf{a},\mathsf{b}}
=\boldsymbol{0}\;\Rightarrow\;\big[\,(a,b^*)\in\boldsymbol{Z}
\;\text{and}\;\eta_{\mathsf{a},\mathsf{b}}=0\,\big]$.
\item
\label{pgre76g12-12iiib}
$\boldsymbol{Z}=\bigcap_{\mathsf{a}\in\gra A}
\bigcap_{\mathsf{b}\in\gra B}
\boldsymbol{H}_{\mathsf{a},\mathsf{b}}$\,.
\end{enumerate}
\item
\label{pgre76g12-12iv}
Let $(a_n,a_n^*)_{n\in\NN}$ be a sequence in $\gra A$, let 
$(b_n,b_n^*)_{n\in\NN}$ be a sequence in $\gra B$, let $x\in\HH$, 
and let $v^*\in\GG$. Suppose that $a_n\weakly x$, 
$b_n^*\weakly v^*$, $a_n^*+L^*b_n^*\to 0$, and $La_n-b_n\to 0$. 
Then $(x,v^*)\in\boldsymbol{Z}$.
\end{enumerate}
\end{proposition}
\begin{proof}
\ref{pgre76g12-12i}: \cite[Proposition~2.8(i)]{Siop11}.

\ref{pgre76g12-12ii}: \cite[Proposition~2.8(iii)-(v)]{Siop11};
see also \cite{Penn00}.

\ref{pgre76g12-12iii}: 
For every $\mathsf{a}=(a,a^*)\in\gra A$ and 
$\mathsf{b}=(b,b^*)\in\gra B$, set
\begin{equation}
\label{eferw7er3h06-13a}
\boldsymbol{G}_{\mathsf{a},\mathsf{b}}=
\menge{\boldsymbol{x}\in\HH\oplus\GG}
{\scal{\boldsymbol{x}}{\boldsymbol{s}_{\mathsf{a},\mathsf{b}}^*}
\leq\eta_{\mathsf{a},\mathsf{b}}},
\end{equation}
and observe that
\begin{align}
\label{eferw7er3h06-13b}
\boldsymbol{H}_{\mathsf{a},\mathsf{b}}
&=\menge{\boldsymbol{x}\in\KKK}
{\scal{\boldsymbol{x}}{\boldsymbol{t}_{\mathsf{a},\mathsf{b}}^*}
\leq\eta_{\mathsf{a},\mathsf{b}}}
\nonumber\\
&=\menge{\boldsymbol{x}\in\KKK}
{\scal{\boldsymbol{x}}
{P_{\KKK}\boldsymbol{s}_{\mathsf{a},\mathsf{b}}^*}
+\scal{\boldsymbol{x}}
{P_{\KKK^\bot}\boldsymbol{s}_{\mathsf{a},\mathsf{b}}^*}
\leq\eta_{\mathsf{a},\mathsf{b}}}
\nonumber\\
&=\menge{\boldsymbol{x}\in\KKK}
{\scal{\boldsymbol{x}}
{\boldsymbol{s}_{\mathsf{a},\mathsf{b}}^*}
\leq\eta_{\mathsf{a},\mathsf{b}}}
\nonumber\\
&=\KKK\cap\boldsymbol{G}_{\mathsf{a},\mathsf{b}}.
\end{align}

\ref{pgre76g12-12iiia}:
By \cite[Proposition~2.2(i)]{Siop14},
$\boldsymbol{G}_{\mathsf{a},\mathsf{b}}=\HH\oplus\GG
\;\Leftrightarrow\;\boldsymbol{s}^*_{\mathsf{a},\mathsf{b}}
=\boldsymbol{0}\;\Rightarrow\;(a,b^*)\in\boldsymbol{Z}
\;\text{and}\;\eta_{\mathsf{a},\mathsf{b}}=0$.
The claim therefore follows from \eqref{eferw7er3h06-13b}.

\ref{pgre76g12-12iiib}:
By \cite[Proposition~2.2(iii)]{Siop14}
$\boldsymbol{Z}=\bigcap_{\mathsf{a}\in\gra A}
\bigcap_{\mathsf{b}\in\gra B}
\boldsymbol{G}_{\mathsf{a},\mathsf{b}}$\,. Hence, 
\eqref{eferw7er3h06-13b} yields
$\boldsymbol{Z}=\KKK\cap\boldsymbol{Z}=
\bigcap_{\mathsf{a}\in\gra A}\bigcap_{\mathsf{b}\in\gra B}
\boldsymbol{H}_{\mathsf{a},\mathsf{b}}$\,. 

\ref{pgre76g12-12iv}: \cite[Proposition~2.4]{Siop14}.
\end{proof}

\begin{remark}
\label{r:KK}
As will be seen in Remark~\ref{rferw7er3h07-12b}, the subspace $\KKK$ 
in Proposition~\ref{pgre76g12-12} adds flexibility to the 
implementation of our proposed algorithms when certain 
structures are present in the problem formulation. 
\end{remark}

\begin{proposition}
\label{paHE5j-7h502}
Problem~\ref{prob:2} is a special case of Problem~\ref{prob:1}.
\end{proposition}
\begin{proof}
Let us set 
\begin{equation}
\label{eccx23gr28s}
\begin{cases}
\HH=\bigoplus_{i\in I}\HH_i\\ 
\GG=\bigoplus_{k\in K}\GG_k\\ 
L\colon\HH\to\GG\colon (x_i)_{i\in I}\mapsto
\big(\sum_{i\in I}L_{ki}x_i\big)_{k\in K}\\
A\colon\HH\to 2^{\HH}\colon(x_i)_{i\in I}\mapsto
\cart_{\!i\in I}(-z^*_i+A_ix_i)\\
B\colon\GG\to 2^{\GG}\colon(y_k)_{k\in K}\mapsto
\cart_{\!k\in K}B_k(y_k-r_k).
\end{cases}
\end{equation}
Then
\begin{equation}
\label{eccx23gr28m}
L^*\colon\GG\to\HH\colon(y_k)_{k\in K}\mapsto
\bigg(\sum_{k\in K}L^*_{ki}y_k\bigg)_{i\in I}.
\end{equation}
With these settings, \eqref{e:primal}, \eqref{e:dual},
and~\eqref{e:2013-05-21a} are respectively equivalent
to~\eqref{eferw7er3h02-26p}, \eqref{eferw7er3h02-26d}, and~\eqref{eferw7er3h02-26k}.
\end{proof}

\subsection{A Fej\'er monotone algorithm}
\label{sec:fejer}

We first recall some basic results concerning Fej\'er monotone sequences.

\begin{proposition}{\rm\cite{Eoop01}}
\label{p:1999}
Let $\KKK$ be a real Hilbert space, let $\boldsymbol{C}$ be a
nonempty closed convex subset of $\KKK$, and let 
$\boldsymbol{x}_0\in\KKK$. Suppose that
\begin{equation}
\label{e:2013-05-20e}
\begin{array}{l}
\text{for}\;n=0,1,\ldots\\
\left\lfloor
\begin{array}{l}
\text{$\boldsymbol{t}_n^*\in\KKK$ and $\eta_n\in\RR$ are
such that $\boldsymbol{C}\subset\boldsymbol{H}_n
=\menge{\boldsymbol{x}\in\KKK}
{\scal{\boldsymbol{x}}{\boldsymbol{t}_n^*}\leq\eta_n}$}\\
\lambda_n\in\left]0,2\right[\\
\boldsymbol{x}_{n+1}=\boldsymbol{x}_n+
\lambda_n(P_{\boldsymbol{H}_n}\boldsymbol{x}_n-\boldsymbol{x}_n).
\end{array}
\right.\\
\end{array}
\end{equation}
Then the following hold:
\begin{enumerate}
\item
\label{p:1999i}
$(\boldsymbol{x}_n)_{n\in\NN}$ is Fej\'er monotone with 
respect to $\boldsymbol{C}$:
$(\forall\boldsymbol{z}\in\boldsymbol{C})(\forall n\in\NN)$
$\|\boldsymbol{x}_{n+1}-\boldsymbol{z}\|\leq
\|\boldsymbol{x}_n-\boldsymbol{z}\|$.
\item
\label{p:1999ii}
$\sum_{n\in\NN}\lambda_n(2-\lambda_n)\|P_{\boldsymbol{H}_n}
\boldsymbol{x}_{n}-\boldsymbol{x}_n\|^2<\pinf$.
\item
\label{p:1999iii}
Suppose that, for every $\boldsymbol{x}\in\KKK$ and every 
strictly increasing sequence $(q_n)_{n\in\NN}$ in $\NN$, 
$\boldsymbol{x}_{q_n}\weakly\boldsymbol{x}$
$\Rightarrow$ $\boldsymbol{x}\in\boldsymbol{C}$.
Then $(\boldsymbol{x}_n)_{n\in\NN}$ converges weakly to a point in
$\boldsymbol{C}$.
\end{enumerate}
\end{proposition}

\begin{algorithm}
\label{algo:1}
Consider the setting of Problem~\ref{prob:1} and let $\KKK$ 
be a closed vector subspace of $\HH\oplus\GG$ such that
$\boldsymbol{Z}\subset\KKK$. Let $\varepsilon\in\zeroun$,
let $(x_0,v_0^*)\in\KKK$, and let
$(\lambda_n)_{n\in\NN}\in[\varepsilon,2-\varepsilon]^{\NN}$.
Iterate
\begin{equation}
\label{e:2013-05-22e}
\begin{array}{l}
\text{for}\;n=0,1,\ldots\\
\left\lfloor
\begin{array}{l}
(a_n,a_n^*)\in\gra A\\
(b_n,b_n^*)\in\gra B\\
(t^*_n,t_n)=P_{\KKK}(a^*_{n}+L^*b^*_{n},b_n-La_n)\\
\tau_n={\|t^*_n\|^2+\|t_n\|^2}\\
\text{if}\;\tau_n>0\\
\left\lfloor
\begin{array}{l}
\theta_n = \Frac{\lambda_n}{\tau_n} \,
\text{\rm max} \Bigbraces{
  0,\scal{x_n}{t^*_n}+\scal{t_n}{v^*_n}
  -\scal{a_n}{a^*_n}-\scal{b_n}{b^*_n}
} 
\end{array} 
\right. \\
\text{else}\; \theta_n = 0 \\
x_{n+1}=x_n-\theta_n t^*_n\\
v^*_{n+1}=v^*_n-\theta_nt_n.
\end{array}
\right.\\[4mm]
\end{array}
\end{equation}
\end{algorithm}

\begin{proposition}
\label{p:2013-05-22}
Consider the setting of Problem~\ref{prob:1} and 
Algorithm~\ref{algo:1}, and suppose that
$\mathscr{P}\neq\emp$. Then the following hold:
\begin{enumerate}
\item
\label{p:2013-05-22i}
$(x_n,v_n^*)_{n\in\NN}$ is a sequence in $\KKK$ which is
Fej\'er monotone with respect to $\boldsymbol{Z}$.
\item
\label{p:2013-05-22iii}
$\sum_{n\in\NN}\|x_{n+1}-x_n\|^2<\pinf$ and
$\sum_{n\in\NN}\|v^*_{n+1}-v^*_n\|^2<\pinf$.
\item
\label{p:2013-05-22iv}
Suppose that the sequences $(a_n)_{n\in\NN}$, $(b_n)_{n\in\NN}$, 
$(a^*_n)_{n\in\NN}$, and $(b^*_n)_{n\in\NN}$ are bounded. Then
\begin{equation}
\label{eBnd6Rfdr01a}
\varlimsup\big(\scal{x_n-a_n}{a_n^*+L^*v_n^*}
+\scal{Lx_n-b_n}{b_n^*-v_n^*}\big)\leq 0.
\end{equation}
\item
\label{p:2013-05-22v}
Suppose that, for every $(x,v^*)\in\KKK$ and for every 
strictly increasing sequence $(q_n)_{n\in\NN}$ in $\NN$, 
\begin{equation}
\label{e:2013-11-27a}
\big[\:x_{q_n}\weakly x\;\;\text{and}\;\;v^*_{q_n}\weakly v^*\:\big]
\quad\Rightarrow\quad (x,v^*)\in\boldsymbol{Z}.
\end{equation}
Then $(x_n)_{n\in\NN}$ converges weakly to a point 
$\overline{x}\in\mathscr{P}$, $(v_n^*)_{n\in\NN}$ converges 
weakly to a point $\overline{v}^*\in\mathscr{D}$, and 
$(\overline{x},\overline{v}^*)\in\boldsymbol{Z}$.
\end{enumerate}
\end{proposition}
\begin{proof}
Parts \ref{pgre76g12-12i} and~\ref{pgre76g12-12ii} of
Proposition~\ref{pgre76g12-12} assert that $\boldsymbol{Z}$ is a
nonempty, closed, and convex subset of $\KKK$. Now set 
\begin{multline}
\label{eferw7er3h06-13t}
(\forall n\in\NN)\quad \boldsymbol{x}_n=(x_n,v_n^*),\;\;
\boldsymbol{s}_n^*=(s_n^*,s_n)=(a^*_n+L^*b^*_n,b_n-La_n),\;\;
\boldsymbol{t}_n^*=(t_n^*,t_n),\\
\eta_n=\scal{a_n}{a^*_n}+\scal{b_n}{b^*_n},\;\;
\text{and}\quad
\boldsymbol{H}_n=\menge{\boldsymbol{x}\in\KKK}
{\scal{\boldsymbol{x}}{\boldsymbol{t}_n^*}\leq\eta_n}. 
\end{multline}
Then it follows from \eqref{e:2013-05-22e} and
Proposition~\ref{pgre76g12-12}\ref{pgre76g12-12iiib} that 
$(\forall n\in\NN)$ $\boldsymbol{Z}\subset\boldsymbol{H}_n$. 
Set $(\forall n\in\NN)$ 
$\Delta_n=\sqrt{\tau_n}\theta_n/\lambda_n$. 
Using \cite[Example~28.16(iii)]{Livre1}, we get
\begin{equation}
\label{eferw7er3h06-13c}
(\forall n\in\NN)\quad P_{\boldsymbol{H}_n}\boldsymbol{x}_{n}=
\begin{cases}
\boldsymbol{x}_n+
\Frac{\eta_n-\scal{\boldsymbol{x}_n}{\boldsymbol{t}^*_n}}
{\|\boldsymbol{t}^*_n\|^2}\boldsymbol{t}^*_n,
&\text{if}\;\;\boldsymbol{t}_n^*\neq
\boldsymbol{0}\;\text{and}\;
\scal{\boldsymbol{x}_n}{\boldsymbol{t}^*_n}>\eta_n;\\
\boldsymbol{x}_n,&\text{otherwise.}
\end{cases}
\end{equation}
Hence, 
\begin{equation}
\label{e:2013-11-29a}
(\forall n\in\NN)\quad\Delta_n=
\|P_{\boldsymbol{H}_n}\boldsymbol{x}_{n}-\boldsymbol{x}_n\|
\quad\text{and}\quad\boldsymbol{x}_{n+1}=\boldsymbol{x}_n+
\lambda_n(P_{\boldsymbol{H}_n}\boldsymbol{x}_n-\boldsymbol{x}_n).
\end{equation}
Therefore, we derive from 
Proposition~\ref{p:1999}\ref{p:1999ii} that
\begin{equation}
\label{eferw7er3h06-14a}
\sum_{n\in\NN}\Delta_n^2<\pinf.
\end{equation}

\ref{p:2013-05-22i}: This follows from \eqref{e:2013-11-29a} and
Proposition~\ref{p:1999}\ref{p:1999i}.

\ref{p:2013-05-22iii}: 
We derive from \eqref{e:2013-05-22e} that 
\begin{equation}
(\forall n\in\NN)\quad
\|x_{n+1}-x_n\|^2+\|v^*_{n+1}-v^*_n\|^2
=\theta_n^2\tau_n=\lambda_n^2\Delta_n^2\leq 4\Delta_n^2.
\end{equation}
Hence, the claim follows from \eqref{eferw7er3h06-14a}.

\ref{p:2013-05-22iv}: 
Since $\|P_{\KKK}\|\leq 1$, \eqref{e:2013-05-22e} and
\eqref{eferw7er3h06-13t} yield
\begin{align}
\label{eferw7er3h03-01c}
(\forall n\in\NN)\quad\tau_n
&=\|\boldsymbol{t}^*_n\|^2\nonumber\\
&\leq\|\boldsymbol{s}^*_n\|^2\nonumber\\
&=\|a_n^*+L^*b_n^*\|^2+\|La_n-b_n\|^2\nonumber\\
&\leq 2\big(\|a^*_n\|^2+\|L\|^2\,\|b^*_n\|^2+
\|L\|^2\|a_n\|^2+\|b_n\|^2\big).
\end{align}
Hence, $(\tau_n)_{n\in\NN}$ is bounded. Therefore, since
\eqref{eferw7er3h06-14a} implies that $\Delta_n\to 0$ and since
$(\boldsymbol{x}_n)_{n\in\NN}$ lies in $\KKK$, we obtain
\begin{align}
\label{eferw7er3h03-01d}
&\hskip 0mm
(\forall n\in\NN)\quad
\text{\rm max}\big\{0,(\scal{x_n}{s^*_n}+\scal{s_n}{v^*_n}
-\scal{a_n}{a^*_n}-\scal{b_n}{b^*_n})\big\}
\nonumber\\
&\hskip 25mm=\text{\rm max}
\big\{0,(\scal{\boldsymbol{x}_n}{\boldsymbol{s}^*_n}
-\scal{a_n}{a^*_n}-\scal{b_n}{b^*_n})\big\}
\nonumber\\
&\hskip 25mm=\text{\rm max}\big\{0,(\scal{\boldsymbol{x}_n}
{P_{\KKK}\boldsymbol{s}^*_n}-\scal{a_n}{a^*_n}-
\scal{b_n}{b^*_n})\big\}
\nonumber\\
&\hskip 25mm=
\text{\rm max}\big\{0,(\scal{x_n}{t^*_n}+\scal{t_n}{v^*_n}
-\scal{a_n}{a^*_n}-\scal{b_n}{b^*_n})\big\}\nonumber\\
&\hskip 25mm=\sqrt{\tau_n}\Delta_n\nonumber\\
&\hskip 25mm\to 0. 
\end{align}
Consequently,
\begin{multline}
\label{eferw7er3h03-01b}
\varlimsup\big(\scal{x_n-a_n}{a_n^*+L^*v_n^*}
+\scal{Lx_n-b_n}{b_n^*-v_n^*}\big)\\
=\varlimsup\big(\scal{x_n}{s^*_n}+\scal{s_n}{v^*_n}
-\scal{a_n}{a^*_n}-\scal{b_n}{b^*_n}\big)
\leq 0.
\end{multline}

\ref{p:2013-05-22v}: This follows from \eqref{e:2013-11-29a} and
Proposition~\ref{p:1999}\ref{p:1999iii}.
\end{proof}

\subsection{An Haugazeau-like algorithm}
\label{sec:haug}

Algorithm~\ref{algo:1} produces sequences that converge weakly to 
some undetermined point in $\boldsymbol{Z}$. We now describe
an algorithm that provides strong convergence to the point in 
$\boldsymbol{Z}$ closest to some reference point 
$(x_0,v_0^*)\in\HH\oplus\GG$. This approach relies on a geometric
construction going back to \cite{Haug68} and was used in 
the context of Problem~\ref{prob:2} in \cite{Nfao15}.

Let $(\boldsymbol{x},\boldsymbol{y},\boldsymbol{z})\in\KKK^3$ be
an ordered triplet from a real Hilbert space $\KKK$. We define
\begin{equation}
H(\boldsymbol{x},\boldsymbol{y})=\menge{\boldsymbol{h}\in\KKK}
{\scal{\boldsymbol{h}-\boldsymbol{y}}
{\boldsymbol{x}-\boldsymbol{y}}\leq 0}
\end{equation}
and, if the set $H(\boldsymbol{x},\boldsymbol{y})\cap
H(\boldsymbol{y},\boldsymbol{z})$ is nonempty, we denote by
$Q(\boldsymbol{x},\boldsymbol{y},\boldsymbol{z})$ the 
projection of $\boldsymbol{x}$ onto it. The principle 
of the algorithm to project a point $\boldsymbol{x}_0\in\KKK$ onto 
a nonempty closed convex set $\boldsymbol{C}\subset\KKK$ is to 
use at iteration $n$ the current iterate $\boldsymbol{x}_n$ to 
construct an outer approximation to $\boldsymbol{C}$ of the form 
$H(\boldsymbol{x}_0,\boldsymbol{x}_n)\cap 
H(\boldsymbol{x}_n,\boldsymbol{x}_{n+1/2})$; the update is then 
computed as the projection of $\boldsymbol{x}_0$ onto this intersection,
i.e., $\boldsymbol{x}_{n+1}=Q(\boldsymbol{x}_0,\boldsymbol{x}_n,
\boldsymbol{x}_{n+1/2})$. As the following lemma from~\cite{Haug68} shows, this
last computation is straightforward; an alternative derivation may be found 
in~\cite[Corollary~28.21]{Livre1}.

\begin{lemma}{\rm(\cite[Th\'eor\`eme~3-1]{Haug68})}
\label{l:haugazeauy}
Let $\KKK$ be a real Hilbert space, let
$(\boldsymbol{x},\boldsymbol{y},\boldsymbol{z})\in\KKK^3$,
and set $\boldsymbol{R}=H(\boldsymbol{x},\boldsymbol{y})\cap
H(\boldsymbol{y},\boldsymbol{z})$. Further, set
$\chi=\scal{\boldsymbol{x}-\boldsymbol{y}}{\boldsymbol{y}-
\boldsymbol{z}}$, $\mu=\|\boldsymbol{x}-\boldsymbol{y}\|^2$,
$\nu=\|\boldsymbol{y}-\boldsymbol{z}\|^2$, and
$\rho=\mu\nu-\chi^2$.
Then exactly one of the following holds:
\begin{enumerate}
\item
\label{c:haugazeaui}
$\rho=0$ and $\chi<0$, in which case $\boldsymbol{R}=\emp$.
\item 
\label{c:haugazeauii}
\emph{[}$\,\rho=0$ and $\chi\geq 0\,$\emph{]} or 
$\rho>0$, in which case $\boldsymbol{R}\neq\emp$ and 
\begin{equation}
\label{e:santiago2014-01-05}
Q(\boldsymbol{x},\boldsymbol{y},\boldsymbol{z})=
\begin{cases}
\boldsymbol{z}, &\!\text{if}\;\rho=0\;\text{and}\;
\chi\geq 0;\\[+0mm]
\displaystyle
\boldsymbol{x}+(1+\chi/\nu)
(\boldsymbol{z}-\boldsymbol{y}), 
&\!\text{if}\;\rho>0\;\text{and}\;
\chi\nu\geq\rho;\\
\displaystyle \boldsymbol{y}+(\nu/\rho)
\big(\chi(\boldsymbol{x}-\boldsymbol{y})
+\mu(\boldsymbol{z}-\boldsymbol{y})\big), 
&\!\text{if}\;\rho>0\;\text{and}\;\chi\nu<\rho.
\end{cases}
\end{equation}
\end{enumerate}
\end{lemma}

\begin{proposition}{\rm(\cite[Proposition~2.1]{Nfao15})}
\label{p:santiago2014-01-06}
Let $\KKK$ be a real Hilbert space, let $\boldsymbol{C}$ be a 
nonempty closed convex subset of $\KKK$, and let 
$\boldsymbol{x}_0\in\KKK$. Iterate
\begin{equation}
\label{e:2013-05-20f}
\begin{array}{l}
\text{for}\;n=0,1,\ldots\\
\left\lfloor
\begin{array}{l}
\text{take}\;\boldsymbol{x}_{n+1/2}\in\KKK\;\text{such that}\;
\boldsymbol{C}\subset H(\boldsymbol{x}_n,\boldsymbol{x}_{n+1/2})\\
\boldsymbol{x}_{n+1}=Q\big(\boldsymbol{x}_0,\boldsymbol{x}_n,
\boldsymbol{x}_{n+1/2}\big).
\end{array}
\right.\\
\end{array}
\end{equation}
Then the sequence $(\boldsymbol{x}_n)_{n\in\NN}$ is well defined 
and the following hold:
\begin{enumerate}
\item
\label{p:santiago2014-01-06i-}
$(\forall n\in\NN)$ $\|\boldsymbol{x}_n-\boldsymbol{x}_0\|\leq
\|\boldsymbol{x}_{\boldsymbol{n}+1}-\boldsymbol{x}_0\|\leq
\|P_{\boldsymbol{C}}\boldsymbol{x}_0-\boldsymbol{x}_0\|$.
\item
\label{p:santiago2014-01-06i}
$(\forall n\in\NN)$ $\boldsymbol{C}\subset
H(\boldsymbol{x}_0,\boldsymbol{x}_n)\cap
H(\boldsymbol{x}_n,\boldsymbol{x}_{n+1/2})$.
\item
\label{p:santiago2014-01-06ii}
$\sum_{n\in\NN}\|\boldsymbol{x}_{n+1}-
\boldsymbol{x}_n\|^2<\pinf$.
\item
\label{p:santiago2014-01-06iii}
$\sum_{n\in\NN}\|\boldsymbol{x}_{n+1/2}-
\boldsymbol{x}_n\|^2<\pinf$.
\item
\label{p:santiago2014-01-06iv}
Suppose that, for every $\boldsymbol{x}\in\KKK$ and every 
strictly increasing sequence $(q_n)_{n\in\NN}$ in $\NN$, 
$\boldsymbol{x}_{q_n}\weakly\boldsymbol{x}$
$\Rightarrow$ $\boldsymbol{x}\in\boldsymbol{C}$.
Then $\boldsymbol{x}_n\to P_{\boldsymbol{C}}\boldsymbol{x}_0$.
\end{enumerate}
\end{proposition}

\begin{algorithm}
\label{algo:2}
Consider the setting of Problem~\ref{prob:1} and let $\KKK$ 
be a closed vector subspace of $\HH\oplus\GG$ such that
$\boldsymbol{Z}\subset\KKK$. Let $\varepsilon\in\zeroun$, let 
$(x_0,v_0^*)\in\KKK$, and let
$(\lambda_n)_{n\in\NN}\in\left[\varepsilon,1\right]^{\NN}$.
Iterate
\begin{equation}
\label{eP9u66fbG02e}
\begin{array}{l}
\text{for}\;n=0,1,\ldots\\
\left\lfloor
\begin{array}{l}
(a_n,a_n^*)\in\gra A\\
(b_n,b_n^*)\in\gra B\\
(t^*_n,t_n)=P_{\KKK}(a^*_n+L^*b^*_n,b_n-La_n)\\
\tau_n={\|t^*_n\|^2+\|t_n\|^2}\\
\text{if}\;\tau_n>0\\
\left\lfloor
\begin{array}{l}
\theta_n=\Frac{\lambda_n}{\tau_n}\,\text{\rm max}\Bigbraces{0,
\scal{x_n}{t^*_n}+\scal{t_n}{v^*_n}
-\scal{a_n}{a^*_n}-\scal{b_n}{b^*_n}}\\
\end{array}
\right.\\
\text{else}\; \theta_n = 0 \\
x_{n+1/2}=x_n-\theta_n t^*_n\\
v^*_{n+1/2}=v^*_n-\theta_nt_n\\
(x_{n+1},v^*_{n+1})=Q\big((x_{0},v^*_{0}),(x_{n},v^*_{n}),
(x_{n+1/2},v^*_{n+1/2})\big).
\end{array}
\right.\\
\end{array}
\end{equation}
\end{algorithm}

\begin{remark}
\label{rferw7er3h06-04}
Using Lemma~\ref{l:haugazeauy}, the computation of the update 
$(x_{n+1},v^*_{n+1})$ in \eqref{eP9u66fbG02e} can be explicitly
broken into the following steps:
\begin{equation}
\label{eP9u66fbG04e}
\begin{array}{l}
\chi_n=\scal{x_0-x_n}{x_n-x_{n+1/2}}
+\scal{v_0^*-v_n^*}{v_n^*-v_{n+1/2}^*}\\
\mu_n=\|x_0-x_n\|^2+\|v_0^*-v_n^*\|^2\\
\nu_n=\|x_n-x_{n+1/2}\|^2+\|v_n^*-v_{n+1/2}^*\|^2\\
\rho_n=\mu_n\nu_n-\chi_n^2\\
\text{if}\;\rho_n=0\;\text{and}\;\chi_n\geq 0\\
\left\lfloor
\begin{array}{l}
x_{n+1}=x_{n+1/2}\\
v^*_{n+1}=v_{n+1/2}^*
\end{array}
\right.\\
\text{if}\;\rho_n>0\;\text{and}\;\chi_n\nu_n\geq\rho_n\\
\left\lfloor
\begin{array}{l}
x_{n+1}=x_0+(1+\chi_n/\nu_n)(x_{n+1/2}-x_n)\\
v^*_{n+1}=v_0^*+(1+\chi_n/\nu_n)(v_{n+1/2}^*-v_n^*)
\end{array}
\right.\\
\text{if}\;\rho_n>0\;\text{and}\;\chi_n\nu_n<\rho_n\\
\left\lfloor
\begin{array}{l}
x_{n+1}=x_n+(\nu_n/\rho_n)\big(\chi_n(x_0-x_n)
+\mu_n(x_{n+1/2}-x_n)\big)\\
v^*_{n+1}=v_n^*+(\nu_n/\rho_n)\big(\chi_n(v_0^*-v_n^*)
+\mu_n(v_{n+1/2}^*-v_n^*)\big).
\end{array}
\right.\\
\end{array}
\end{equation}
\end{remark}

\begin{proposition}
\label{pP9u66fbG03}
Consider the setting of Problem~\ref{prob:1} and
Algorithm~\ref{algo:2}. Suppose that
$\mathscr{P}\neq\emp$ and set $(\overline{x},\overline{v}^*)=
P_{\boldsymbol{Z}}(x_0,v_0^*)$. 
Then the following hold:
\begin{enumerate}
\item
\label{pP9u66fbG03i}
$(x_n)_{n\in\NN}$ and $(v_n^*)_{n\in\NN}$ are bounded.
\item
\label{pP9u66fbG03ii}
$\sum_{n\in\NN}\|x_{n+1}-x_n\|^2<\pinf$ and
$\sum_{n\in\NN}\|v^*_{n+1}-v^*_n\|^2<\pinf$.
\item
\label{pP9u66fbG03iii}
$\sum_{n\in\NN}\|x_{n+1/2}-x_n\|^2<\pinf$ and
$\sum_{n\in\NN}\|v^*_{n+1/2}-v^*_n\|^2<\pinf$.
\item
\label{pP9u66fbG03iv}
Suppose that the sequences
$(a_n)_{n\in\NN}$, $(b_n)_{n\in\NN}$, $(a^*_n)_{n\in\NN}$, and
$(b^*_n)_{n\in\NN}$ are bounded. Then
\begin{equation}
\label{eP9u66fbG03h}
\varlimsup\big(\scal{x_n-a_n}{a_n^*+L^*v_n^*}
+\scal{Lx_n-b_n}{b_n^*-v_n^*}\big)\leq 0.
\end{equation}
\item
\label{pP9u66fbG03v}
Suppose that, for every $(x,v^*)\in\KKK$ and every 
strictly increasing sequence $(q_n)_{n\in\NN}$ in $\NN$, 
\begin{equation}
\label{eferw7er3h06-03g}
\big[\:x_{q_n}\weakly x\;\;\text{and}\;\;v^*_{q_n}\weakly v^*\:\big]
\quad\Rightarrow\quad (x,v^*)\in\boldsymbol{Z}.
\end{equation}
Then $(x_n)_{n\in\NN}$ converges strongly to 
$\overline{x}\in\mathscr{P}$ and $(v_n^*)_{n\in\NN}$ converges 
strongly to $\overline{v}^*\in\mathscr{D}$.
\end{enumerate}
\end{proposition}
\begin{proof}
We first show that we recover the setting of 
Proposition~\ref{p:santiago2014-01-06} applied in $\KKK$ to the
set $\boldsymbol{Z}$ of \eqref{e:2013-05-21a}, which is
nonempty, closed, and convex by
Proposition~\ref{pgre76g12-12}\ref{pgre76g12-12i}%
--\ref{pgre76g12-12ii}. Set
\begin{multline}
\label{e:santiago2014-01-11b}
(\forall n\in\NN)\quad
\boldsymbol{x}_n=(x_n,v_n^*),\quad
\boldsymbol{x}_{n+1/2}=(x_{n+1/2},v_{n+1/2}^*),\quad
\boldsymbol{t}^*_n=(t^*_n,t_n),\\
\text{and}\quad\eta_n=\scal{a_n}{a^*_n}+\scal{b_n}{b^*_n}.
\end{multline}
If, for some $n\in\NN$, we have
$\boldsymbol{x}_{n+1/2}=\boldsymbol{x}_n$, then trivially 
$\boldsymbol{Z}\subset H(\boldsymbol{x}_n,\boldsymbol{x}_{n+1/2})
=\KKK$; otherwise, \eqref{eP9u66fbG02e} imposes that 
$\scal{\boldsymbol{x}_n}{\boldsymbol{t}^*_n}>\eta_n$ and 
therefore that
\begin{align}
\label{eP9u66fbG05x}
\eta_n
&\leq\scal{\boldsymbol{x}_n}{\boldsymbol{t}^*_n}-
\lambda_n\big(\scal{\boldsymbol{x}_n}{\boldsymbol{t}^*_n}
-\eta_n\big)\nonumber\\
&=\scal{\boldsymbol{x}_n}{\boldsymbol{t}^*_n}-
\theta_n\tau_n\nonumber\\
&=\scal{\boldsymbol{x}_n-\theta_n\boldsymbol{t}^*_n}
{\boldsymbol{t}^*_n}\nonumber\\
&=\scal{\boldsymbol{x}_{n+1/2}}{\boldsymbol{t}^*_n},
\end{align}
from which we deduce using
Proposition~\ref{pgre76g12-12}\ref{pgre76g12-12iii} that
\begin{align}
\label{eP9u66fbG05a}
\boldsymbol{Z}
&\subset 
\menge{\boldsymbol{x}\in\boldsymbol{\mathcal{K}}}
{\scal{\boldsymbol{x}}{\boldsymbol{t}^*_n}\leq\eta_n}\nonumber\\
&\subset
\menge{\boldsymbol{x}\in\boldsymbol{\mathcal{K}}}
{\scal{\boldsymbol{x}}{\boldsymbol{t}^*_n}
\leq\scal{\boldsymbol{x}_{n+1/2}}
{\boldsymbol{t}^*_n}}\nonumber\\
&=\menge{\boldsymbol{x}\in\boldsymbol{\mathcal{K}}}
{\scal{\boldsymbol{x}}{\boldsymbol{x}_n-\boldsymbol{x}_{n+1/2}}
\leq\scal{\boldsymbol{x}_{n+1/2}}
{\boldsymbol{x}_n-\boldsymbol{x}_{n+1/2}}}\nonumber\\
&=H\big(\boldsymbol{x}_n,\boldsymbol{x}_{n+1/2}\big).
\end{align}
Altogether, \eqref{eP9u66fbG02e} is an instance of 
\eqref{e:2013-05-20f} with $\boldsymbol{C}=\boldsymbol{Z}$, and 
we can apply Proposition~\ref{p:santiago2014-01-06}. In particular,
Proposition~\ref{p:santiago2014-01-06}\ref{p:santiago2014-01-06i} 
asserts that $(x_n,v_n^*)_{n\in\NN}$ is well defined.  
We can now establish the claims of
the proposition as follows.

\ref{pP9u66fbG03i}: This is a consequence of 
Proposition~\ref{p:santiago2014-01-06}\ref{p:santiago2014-01-06i-}.

\ref{pP9u66fbG03ii}:
It follows from \eqref{e:santiago2014-01-11b} and 
Proposition~\ref{p:santiago2014-01-06}\ref{p:santiago2014-01-06ii}
that $\sum_{n\in\NN}\|x_{n+1}-x_n\|^2+
\sum_{n\in\NN}\|v^*_{n+1}-v^*_n\|^2=
\sum_{n\in\NN}\|\boldsymbol{x}_{n+1}-
\boldsymbol{x}_n\|^2<\pinf$.

\ref{pP9u66fbG03iii}:
In view of \eqref{e:santiago2014-01-11b} and 
Proposition~\ref{p:santiago2014-01-06}\ref{p:santiago2014-01-06iii},
$\sum_{n\in\NN}\|x_{n+1/2}-x_n\|^2+
\sum_{n\in\NN}\|v^*_{n+1/2}-v^*_n\|^2=
\sum_{n\in\NN}\|\boldsymbol{x}_{n+1/2}-
\boldsymbol{x}_n\|^2<\pinf$.

\ref{pP9u66fbG03iv}: 
Set $(\forall n\in\NN)$ $\Delta_n=\sqrt{\tau_n}\theta_n/\lambda_n$. 
We derive from 
\eqref{eP9u66fbG02e} and \ref{pP9u66fbG03iii} that
\begin{equation}
\label{eP9u66fbG05y}
\sum_{n\in\NN}\Delta_n^2=
\sum_{n\in\NN}\frac{\tau_n\theta_n^2}{\lambda_n^2}
\leq\sum_{n\in\NN}\frac{\tau_n\theta^2_n}{\varepsilon^2}
=\sum_{n\in\NN}\frac{\|\boldsymbol{x}_{n+1/2}-\boldsymbol{x}_n\|^2}
{\varepsilon^2}<\pinf.
\end{equation}
The claim is then obtained by arguing as in the proof of 
Proposition~\ref{p:2013-05-22}\ref{p:2013-05-22iv}.

\ref{pP9u66fbG03v}:
This follows directly from 
Proposition~\ref{p:santiago2014-01-06}\ref{p:santiago2014-01-06iv}.
\end{proof}

\begin{remark}
\label{rferw7er3h07-15x}
Proposition~\ref{pP9u66fbG03} guarantees strong convergence to
the projection of the initial point $(x_0,v_0^*)$ onto the 
Kuhn-Tucker set under the same conditions that provide weak 
convergence to an unspecified Kuhn-Tucker point in 
Proposition~\ref{p:2013-05-22}. This phenomenon is akin to the
weak-to-strong convergence principle investigated in a fixed-point
setting in \cite{Moor01}.
\end{remark}

\section{Solving Problem~\ref{prob:2}}
\label{sec:3}

\subsection{Block iterations and asynchronicity}

In existing monotone operator splitting methods, each operator
in the inclusion problem must be used at each iteration 
$n$ in a resolvent calculation that must be based on 
information available at the current 
iteration. For instance, the methods of 
\cite{Siop14,Nfao15} require 
points $(a_{i,n},a^*_{i,n})\in\gra A_i$ and 
$(b_{k,n},b^*_{k,n})\in\gra B_k$ for every $i\in I$ and every 
$k\in K$, 
and these points must be computed using the current values of
the primal variables $(x_{i,n})_{i\in I}$ and of the dual variables
$(v^*_{k,n})_{k\in K}$.  The earlier work in~\cite{Svai08,Svai09}
in the context of \eqref{e:2009} is similar. The two
main novelties we present in this paper are to depart from this 
approach by
allowing asynchronous block iterations. Specifically, we allow:

{\bfseries Block iterations:} 
At iteration $n$, we require calculation
of new points in the graphs of only some of the operators, say
$(A_i)_{i\in I_n\subset I}$ and $(B_k)_{k\in K_n\subset K}$. The
control sequences $(I_n)_{n\in\NN}$ and $(K_n)_{n\in\NN}$ dictate
how frequently the various operators are used.

{\bfseries Asynchronicity:} 
A new point $(a_{i,n},a^*_{i,n})\in\gra A_i$ being incorporated
into the calculations at iteration $n$ may be based on data
$x_{i,c_i(n)}$ and $(v^*_{k,c_i(n)})_{k\in K}$ available at 
some possibly earlier iteration $c_i(n)\leq n$. Therefore, the
calculation of $(a_{i,n},a^*_{i,n})$ could have been initiated at
iteration $c_i(n)$, with its results becoming available only at
iteration $n$. Likewise, for every $k\in K_n$, the computation of
$(b_{k,n},b^*_{k,n})\in\gra B_k$ can be initiated at some iteration
$d_k(n)\leq n$, based on $(x_{i,d_k(n)})_{i\in I}$ and
$v^*_{k,d_k(n)}$.

To establish convergence, there needs to be some limits on the
asynchronous asynchronicity lag of the algorithm and the spacing 
between successive calculations involving each operator, as 
described in the following assumption.

\begin{assumption} \
\label{ass:validcontrol}
\begin{enumerate}
\item 
$M$ is a strictly positive integer, $(I_n)_{n\in\NN}$ is a sequence 
of nonempty subsets of $I$, and $(K_n)_{n\in\NN}$ is a sequence of 
nonempty subsets of $K$ such that
\begin{equation}
\label{ekjui7743v10-24G}
I_0=I,\;K_0=K,\quad\text{and}\quad
(\forall n\in\NN) \;\; \left( \bigcup_{j=n}^{n+M-1}I_j=I
\quad\text{and}\quad\bigcup_{j=n}^{n+M-1}K_j=K \right).
\end{equation}
\item
$D$ is a positive integer and, for every $i\in I$ and every 
$k\in K$, $(c_i(n))_{n\in\NN}$ and $(d_k(n))_{n\in\NN}$ are 
sequences in $\NN$ such that
\begin{equation} 
\label{e:lag}
(\forall n\in\NN)\quad
\Big((\forall i\in I) \;\;n-D \leq c_i(n) \leq n \
\quad \text{and} \quad
(\forall k \in K)\;\;n-D\leq d_k(n)\leq n\Big).
\end{equation}
\item
\label{ass:validcontroliii}
$\varepsilon\in\zeroun$ and, for every $i\in I$ and every $k\in K$, 
$(\gamma_{i,n})_{n\in\NN}$ and $(\mu_{k,n})_{n\in\NN}$ are 
sequences in $[\varepsilon,1/\varepsilon]$.
\end{enumerate}
\end{assumption}

At iteration $n$, our algorithms incorporates points in the
graphs of the operators $(A_i)_{i\in I_n}$ and
$(B_k)_{k\in K_n}$. Condition~\eqref{ekjui7743v10-24G} ensures 
that over any span of $M$ consecutive iterations, each operator is 
incorporated into the algorithm at least once. The standard case
corresponds to using all the operators at each iteration, \emph{i.e.}
$(\forall n\in\NN)$ $I_n=I$ and $K_n=K$. Toward the other extreme, 
it is possible to use just one of the operators from 
$(A_i)_{i\in I}$ and $(B_k)_{k\in K}$ at iteration $n$. 
For example, such
a control regime could be achieved by setting $M=\text{max}\{m,p\}$ 
and sweeping though the operators in a periodic manner.
Condition~\eqref{e:lag} guarantees that 
the points in the graphs incorporated into the algorithm are
based on information at most $D$ iterations out
of date. 
If the algorithm is being implemented synchronously, then one can 
simply set $D=0$, in which case 
$(\forall n\in\NN)(\forall i\in I)(\forall k\in K)$
$c_i(n)=n$ and $d_k(n)=n$. Finally, the positive scalars 
$(\gamma_{i,n})_{n\in\NN}$ and
$(\mu_{k,n})_{n\in\NN}$ in \ref{ass:validcontroliii} are the proximal
parameters used in the resolvent calculations. The assumption
requires that they be bounded above and also away from $0$.

The following result is the key asymptotic principle on which our
two main theorems will rest.  The key idea of our algorithm is to
simply recycle an old point in the graph of each operator for
which new information is not available.

\begin{proposition}
\label{pP9u66fbG07}
Consider the setting of Problem~\ref{prob:2} and suppose that the 
following are satisfied: 
\begin{enumerate}[label=\rm(\alph*)]
\item
\label{aP9u66fbG07i}
For every $i\in I$, $(x_{i,n})_{n\in\NN}$ is a bounded sequence in 
$\HH_i$ and, for every $k\in K$, $(v^*_{k,n})_{n\in\NN}$ is a 
bounded sequence in $\GG_k$.
\item
\label{aP9u66fbG07ii}
Assumption~\ref{ass:validcontrol} is in force.
\item
\label{aP9u66fbG07iii}
For every $n\in\NN$, set
\begin{equation}
\label{eP9u66fbG07a}
\begin{array}{l}
\text{for every}\;i\in I_n\\
\left\lfloor
\begin{array}{l}
l^*_{i,n}=\sum_{k\in K}L_{ki}^*v_{k,c_i(n)}^*\\
(a_{i,n},a_{i,n}^*)=
\Big(J_{\gamma_{i,c_i(n)} A_i}\big(x_{i,c_i(n)}+\gamma_{i,c_i(n)}
(z^*_i-l^*_{i,n})\big),\gamma_{i,c_i(n)}^{-1}
(x_{i,c_i(n)}-a_{i,n})-l^*_{i,n}\Big)\\
\end{array}
\right.\\[1mm]
\text{for every}\;i\in I\smallsetminus I_n\\
\left\lfloor
\begin{array}{l}
(a_{i,n},a_{i,n}^*)=(a_{i,n-1},a_{i,n-1}^*)\\
\end{array}
\right.\\[1mm]
\text{for every}\;k\in K_n\\
\left\lfloor
\begin{array}{l}
l_{k,n}=\sum_{i\in I}L_{ki}x_{i,d_k(n)}\\
(b_{k,n},b^*_{k,n})=\Big(r_k+J_{\mu_{k,d_k(n)}B_k}
\big(l_{k,n}+\mu_{k,d_k(n)}v_{k,d_k(n)}^*-r_k\big),
v_{k,d_k(n)}^*+\mu_{k,d_k(n)}^{-1}(l_{k,n}-b_{k,n})\Big)\\
\end{array}
\right.\\[1mm]
\text{for every}\;k\in K\smallsetminus K_n\\
\left\lfloor
\begin{array}{l}
(b_{k,n},b^*_{k,n})=(b_{k,n-1},b^*_{k,n-1}),\\
\end{array}
\right.\\[-5mm]
\end{array}
\end{equation}
and define
\begin{equation}
\label{eccx23gr28t}
(\forall n\in\NN)\quad
a_n=(a_{i,n})_{i\in I},\;\;
a_n^*=(a_{i,n}^*)_{i\in I},\;\;
b_n=(b_{k,n})_{k\in K},\;\;\text{and}\;\;
b_n^*=(b_{k,n}^*)_{k\in K}.
\end{equation}
\end{enumerate}
Then the following hold:
\begin{enumerate}
\item
\label{pP9u66fbG07i}
Define $A$ and $B$ as in 
\eqref{eccx23gr28s}. Then
$(\forall n\in\NN)$ $(a_n,a_n^*)\in\gra A$ \;and\;
$(b_n,b_n^*)\in\gra B$.
\item
\label{pP9u66fbG07ii}
$(a_{n})_{n\in\NN}$,\;
$(a^*_{n})_{n\in\NN}$,\;
$(b_{n})_{n\in\NN}$,\; and\; 
$(b^*_{n})_{n\in\NN}$\; are bounded. 
\item
\label{pP9u66fbG07iii}
Suppose that the following are satisfied:
\begin{enumerate}
\setcounter{enumii}{3}
\item
\label{aP9u66fbG07iv}
$(\forall i\in I)$ $\sum_{n\in\NN}\|x_{i,n+1}-x_{i,n}\|^2<\pinf$ 
and $(\forall k\in K)$
$\sum_{n\in\NN}\|v^*_{k,n+1}-v^*_{k,n}\|^2<\pinf$.
\item
\label{aP9u66fbG07v}
$\varlimsup\big(\sum_{i\in I}\scal{x_{i,n}\!-\!a_{i,n}}
{a_{i,n}^*\!+\!\sum_{k\in K}L_{ki}^*v_{k,n}^*}
+\sum_{k\in K}\scal{\sum_{i\in I}L_{ki}x_{i,n}\!-\!b_{k,n}}
{b_{k,n}^*\!-\!v_{k,n}^*}\big)\leq 0$.
\item
\label{aP9u66fbG07vi}
$(q_n)_{n\in\NN}$ is a strictly 
increasing sequence in $\NN$,
for every $i\in I$, $x_i\in\HH_i$ and
$x_{i,q_n}\weakly x_i$,  and, 
for every $k\in K$, $v^*_k\in \GG_k$ and
$v^*_{k,q_n}\weakly v^*_k$.
\end{enumerate}
Then $((x_i)_{i\in I},(v^*_k)_{k\in K})\in\boldsymbol{Z}$.
\end{enumerate}
\end{proposition}
\begin{proof}
Define $\HH$, $\GG$, and $L$ as in \eqref{eccx23gr28s} and set
\begin{equation}
\label{eccx23gr28a}
(\forall n\in\NN)\quad
x_n=(x_{i,n})_{i\in I}\quad\text{and}\quad
v_n^*=(v_{k,n}^*)_{k\in K}.
\end{equation}

\ref{pP9u66fbG07i}: This follows from \eqref{eP9u66fbG07a}
and basic resolvent calculus rules 
\cite[Propositions~23.15 and 23.16]{Livre1}. 

\ref{pP9u66fbG07ii}: 
Let $i\in I$.
We derive from hypothesis~\ref{aP9u66fbG07i} and
Assumption~\ref{ass:validcontrol}\ref{ass:validcontroliii}
that the sequence $\big(x_{i,c_i(n)}-\gamma_{i,c_i(n)}\sum_{k\in K}
L_{ki}^*v_{k,c_i(n)}^*\big)_{n\in\NN}$ is bounded. 
Since the operators $(J_{\gamma_{i,c_i(n)}A_i})_{n\in\NN}$ 
are nonexpansive \cite[Corollary~23.8]{Livre1}, it follows 
from \eqref{eP9u66fbG07a} that $(a_{i,n})_{n\in\NN}$ is 
bounded, and hence that
$(a^*_{i,n})_{n\in\NN}$ is also bounded. Likewise,
for every $k\in K$, 
$(\sum_{i\in I}L_{ki}x_{i,d_k(n)}+\mu_{k,d_k(n)}
v_{k,d_k(n)}^*)_{n\in\NN}$ is bounded and we deduce from
\eqref{eP9u66fbG07a} that $(b_{k,n})_{n\in\NN}$ and 
$(b^*_{k,n})_{n\in\NN}$ are bounded. 
In view of \eqref{eccx23gr28t}, this establishes 
the claim.

\ref{pP9u66fbG07iii}:
For every every $i\in I$ and every $n\in\NN$, define $\bar\ell_i(n)$ 
as the most recent iteration at which a new point in the 
graph of $A_i$ was incorporated into the algorithm, that is,
\begin{equation}
\label{eBnd6Rfdr31a}
(\forall i\in I)(\forall n\in\NN)\quad
\bar\ell_i(n)=\text{max}\menge{j\in S_i}{j\leq n},
\quad\text{where}\quad S_i=\menge{j\in\NN}{i\in I_j}.
\end{equation}
Note that \eqref{eP9u66fbG07a} implies that
\begin{equation}
\label{eBnd6Rfdr31f}
(\forall i\in I)(\forall n\in\NN)\quad(a_{i,n},a^*_{i,n})=
\big(a_{i,\bar\ell_i(n)},a^*_{i,\bar\ell_i(n)}\big).
\end{equation}
For every $i\in I$, \eqref{ekjui7743v10-24G} yields
$\sup_{n\in\NN}(n-\bar\ell_i(n))\leq M$ and hence 
$\lim_{n\rightarrow\pinf}\bar\ell_i(n)=\pinf$.  
Next, we define 
\vspace{-2ex}  
\begin{equation}
\label{e:juin2015a}
(\forall i\in I)(\forall n\in\NN)\quad
\ell_i(n)=c_i\big(\bar\ell_i(n)\big).
\end{equation}
Thus, $\ell_i(n)$ is the iteration from which the computation
of the most recent point in the graph of $A_i$ was initiated.
It follows from \eqref{e:lag} that 
\begin{equation}
(\forall i\in I)(\forall n\in\NN)\quad
n-\ell_i(n)=n-\bar\ell_i(n) +\bar\ell_i(n)-\ell_i(n)\leq M+D.
\end{equation}
Hence, $(\forall i\in I)$ $\lim_{n\to\pinf}\ell_i(n)=\pinf$. Since 
$\text{max}_{i\in I}\sum_{j\in\NN}\|x_{i,j+1}-x_{i,j}\|^2<\pinf$ 
by \ref{aP9u66fbG07iv}, we deduce that
\begin{align}
\label{eaHE5j-7h503h}
(\forall i\in I)\quad 
\|x_{i,n}-x_{i,\ell_i(n)}\|^2
&\leq\left(\sum_{j=\ell_i(n)}^{\ell_i(n)+M+D-1}
\|x_{i,j+1}-x_{i,j}\|\right)^2 \nonumber\\
&\leq (M + D)\sum_{j=\ell_i(n)}^{\ell_i(n)+M+D-1}
\|x_{i,j+1}-x_{i,j}\|^2\nonumber\\
&\leq (M + D) \sum_{j=\ell_i(n)}^{\pinf}\|x_{i,j+1}-x_{i,j}\|^2
\nonumber\\
&\to 0.
\end{align}
Likewise, since \ref{aP9u66fbG07iv} asserts that 
$\text{max}_{k\in K}\sum_{j\in\NN}\|v^*_{k,j+1}-v^*_{k,j}\|^2<\pinf$,
we have 
\begin{equation}
\label{eaHE5j-7h503d}
(\forall i\in I)(\forall k\in K)\quad 
\|v^*_{k,n}-v^*_{k,\ell_i(n)}\|^2\leq
(M+D)\sum_{j=\ell_i(n)}^{\pinf}\|v^*_{k,j+1}-v^*_{k,j}\|^2\to 0.
\end{equation}
Next, let us set 
\begin{equation}
\label{eaHE5j-7h504b}
(\forall i\in I)(\forall n\in\NN)\quad
\begin{cases}
\phi_{i,n}=\Scal{x_{i,n}-a_{i,n}}{a_{i,n}^*+\Sum_{k\in K}
L_{ki}^*v^*_{k,n}}\\[4mm]
\widetilde{\phi}_{i,n}=\Scal{x_{i,\ell_i(n)}-a_{i,\bar\ell_i(n)}}
{a_{i,\bar\ell_i(n)}^*+\Sum_{k\in K}L_{ki}^*v^*_{k,\ell_i(n)}}.
\end{cases}
\end{equation}
Then it follows from \eqref{eBnd6Rfdr31f},
\ref{aP9u66fbG07i}, \ref{pP9u66fbG07ii},
\eqref{eaHE5j-7h503h}, and \eqref{eaHE5j-7h503d} that
\begin{align}
\label{eaHE5j-7h504c}
(\forall i\in I)\quad
\phi_{i,n}-\widetilde{\phi}_{i,n}
&=\Scal{x_{i,n}-a_{i,n}}{a_{i,n}^*+\sum_{k\in K}
L_{ki}^*v^*_{k,n}}\nonumber\\
&\quad\;-\Scal{x_{i,\ell_i(n)}-a_{i,n}}
{a_{i,n}^*+\sum_{k\in K}L_{ki}^*v^*_{k,\ell_i(n)}}
\nonumber\\
&=\Scal{x_{i,n}-a_{i,n}}{\sum_{k\in K}L_{ki}^*
(v^*_{k,n}-v^*_{k,\ell_i(n)})}\nonumber\\
&\quad\;+\Scal{x_{i,n}-x_{i,\ell_i(n)}}
{a_{i,n}^*+\sum_{k\in K}L_{ki}^*v^*_{k,\ell_i(n)}}
\nonumber\\
&\leq\bigg(\sum_{k\in K}\|L_{ki}\|\sup_{j\in\NN}
\big(\|x_{i,j}\|+\|a_{i,j}\|\big)\bigg)
\|v^*_{k,n}-v^*_{k,\ell_i(n)}\|\nonumber\\
&\quad\;+\bigg(\sup_{j\in\NN}\|a_{i,j}^*\|+
\sum_{k\in K}\|L_{ki}\|\sup_{j\in\NN}\|v^*_{k,j}\|
\bigg)\|x_{i,n}-x_{i,\ell_i(n)}\|
\nonumber\\
&\to 0.
\end{align}
We also derive from \eqref{eaHE5j-7h504b},
\eqref{eBnd6Rfdr31f}, and \eqref{eP9u66fbG07a} that
\begin{align}
\label{eaHE5j-7h504r}
(\forall i\in I)(\forall n\in\NN)\quad\widetilde{\phi}_{i,n}
&=\Scal{x_{i,\ell_i(n)}-a_{i,n}}
{a_{i,n}^*+\sum_{k\in K}L_{ki}^*v^*_{k,\ell_i(n)}}\nonumber\\
&=\gamma_{i,\ell_i(n)}^{-1}\|x_{i,\ell_i(n)}-a_{i,n}\|^2\nonumber\\
&=\gamma_{i,\ell_i(n)}\bigg\|a_{i,n}^*+\sum_{k\in K}L_{ki}^*
v^*_{k,\ell_i(n)}\bigg\|^2,
\end{align}
which yields
\begin{equation}
\label{eaHE5j-7h503s}
(\forall i\in I)(\forall n\in\NN)\quad
\widetilde{\phi}_{i,n}=\gamma_{i,\ell_i(n)}^{-1}
\|x_{i,\ell_i(n)}-a_{i,n}\|^2\geq
\varepsilon\|x_{i,\ell_i(n)}-a_{i,n}\|^2
\end{equation}
and
\begin{equation}
\label{eaHE5j-7h503v}
(\forall i\in I)(\forall n\in\NN)\quad
\widetilde{\phi}_{i,n}=
\gamma_{i,\ell_i(n)}\bigg\|a_{i,n}^*+\Sum_{k\in K}L_{ki}^*
v^*_{k,\ell_i(n)}\bigg\|^2
\geq\varepsilon\bigg\|a_{i,n}^*+
\Sum_{k\in K}L_{ki}^*v^*_{k,\ell_i(n)}\bigg\|^2.
\end{equation}
It follows from \eqref{eaHE5j-7h503s} that
\begin{align}
\label{eaHE5j-7h505n}
(\forall i\in I)(\forall n\in\NN)\quad
\|x_{i,n}-a_{i,n}\|^2
&\leq 2\big(\|x_{i,n}-x_{i,\ell_i(n)}\|^2
+\|x_{i,\ell_i(n)}-a_{i,n}\|^2\big)
\nonumber\\
&\leq2\big(\|x_{i,n}-x_{i,\ell_i(n)}\|^2+
\varepsilon^{-1}(\widetilde{\phi}_{i,n}-\phi_{i,n})
+\varepsilon^{-1}\phi_{i,n}\big)
\end{align}
and from \eqref{eaHE5j-7h503v} that
\begin{align}
\label{eaHE5j-7h505m}
&\hskip -6mm (\forall i\in I)(\forall n\in\NN)\;\;
\bigg\|a_{i,n}^*+\Sum_{k\in K}L_{ki}^*v^*_{k,n}\bigg\|^2
\nonumber\\
&\hskip 32mm\leq 
\bigg\|a_{i,n}^*+\Sum_{k\in K}L_{ki}^*v^*_{k,\ell_i(n)}-
\Sum_{k\in K}L_{ki}^*(v^*_{k,\ell_i(n)}-v^*_{k,n})\bigg\|^2
\nonumber\\
&\hskip 32mm\leq 2\bigg(\bigg\|a_{i,n}^*+
\Sum_{k\in K}L_{ki}^*v^*_{k,\ell_i(n)}\bigg\|^2+\Sum_{k\in K}
\|L_{ki}\|^2\,\|v^*_{k,\ell_i(n)}-v^*_{k,n}\|^2\bigg)
\nonumber\\
&\hskip 32mm\leq2\bigg(\varepsilon^{-1}
(\widetilde{\phi}_{i,n}-\phi_{i,n})
+\varepsilon^{-1}\phi_{i,n}
+\Sum_{k\in K}\|L_{ki}\|^2\,\|v^*_{k,\ell_i(n)}-v^*_{k,n}\|^2
\bigg).
\end{align}
We now perform a similar analysis for the operators 
$(B_k)_{k\in K}$. Much as in 
\eqref{eBnd6Rfdr31a}, for every $k\in K$ and every
$n\in\NN$, define $\bar\vartheta_k(n)$ as the most recent 
iteration at which a new point in the graph of $B_k$ was 
incorporated into the algorithm, that is,
\begin{equation}
\label{eccx23gr24A}
(\forall k\in K)(\forall n\in\NN)\quad
\bar\vartheta_k(n)=\text{max}\menge{j\in T_k}{j\leq n},
\quad\text{where}\quad T_k=\menge{j\in\NN}{k\in K_j},
\end{equation}
and observe that
\begin{equation}
\label{eBnd6Rfdr31g}
(\forall k\in K)(\forall n\in\NN)\quad (b_{k,n},b^*_{k,n})=
\big(b_{k,\bar\vartheta_k(n)},b^*_{k,\bar\vartheta_k(n)}\big).
\end{equation}
Next, we define 
\begin{equation}
\label{e:juin2015b}
(\forall k\in K)(\forall n\in\NN)\quad
\vartheta_k(n)=d_k\big(\bar\vartheta_k(n)\big).
\end{equation}
Then, we derive from \eqref{ekjui7743v10-24G} and \eqref{e:lag} that
\begin{equation}
(\forall k\in K)(\forall n\in\NN)\quad n-\vartheta_k(n)=
n-\bar\vartheta_k(n)+\bar\vartheta_k(n)-\vartheta_k(n)\leq M+D
\end{equation}
and therefore that $(\forall k\in K)$ 
$\lim_{k\to\pinf}\vartheta_k(n)=\pinf$.
Since 
$\text{max}_{i\in I}\sum_{j\in\NN}\|x_{i,j+1}-x_{i,j}\|^2<\pinf$ 
by \ref{aP9u66fbG07iv}, we then deduce that 
\begin{equation}
\label{eaHE5j-7h503t}
(\forall k\in K)\quad 
\|x_{i,n}-x_{i,\vartheta_k(n)}\|^2\leq
(M+D)\sum_{j=\vartheta_k(n)}^{\pinf}\|x_{i,j+1}-x_{i,j}\|^2\to 0.
\end{equation}
Similarly since, $\text{max}_{k\in K}
\sum_{j\in\NN}\|v^*_{k,j+1}-v^*_{k,j}\|^2<\pinf$, we have
\begin{equation}
\label{eaHE5j-7h503b}
(\forall k\in K)\quad 
\|v^*_{k,n}-v^*_{k,\vartheta_k(n)}\|^2\leq
(M+D)\sum_{j=\vartheta_k(n)}^{\pinf}\|v^*_{k,j+1}-v^*_{k,j}\|^2\to 0.
\end{equation}
Now, let us set
\begin{equation}
\label{eccx23gr24b}
(\forall k\in K)(\forall n\in\NN)\quad
\begin{cases}
\psi_{k,n}=
\Scal{\Sum_{i\in I}
L_{ki}x_{i,n}-b_{k,n}}{b_{k,n}^*-v_{k,n}^*}\\[4mm]
\widetilde{\psi}_{k,n}=
\Scal{\Sum_{i\in I}L_{ki}x_{i,\vartheta_k(n)}-
b_{k,\bar\vartheta_k(n)}}
{b_{k,\bar\vartheta_k(n)}^*-v_{k,\vartheta_k(n)}^*}.
\end{cases}
\end{equation}
Then it follows from \eqref{eBnd6Rfdr31g},  
\ref{aP9u66fbG07i}, \ref{pP9u66fbG07ii},
\eqref{eaHE5j-7h503t}, and \eqref{eaHE5j-7h503b} that
\begin{align}
\label{eccx23gr24o}
(\forall k\in K)\quad
\psi_{k,n}-\widetilde{\psi}_{k,n}
&=\Scal{\sum_{i\in I}L_{ki}x_{i,n}-b_{k,n}}{b_{k,n}^*-v_{k,n}^*}
\nonumber\\
&\quad\;-\Scal{\sum_{i\in I}L_{ki}x_{i,\vartheta_k(n)}-
b_{k,n}}{b_{k,n}^*
-v_{k,\vartheta_k(n)}^*}\nonumber\\
&=\Scal{\sum_{i\in I}L_{ki}(x_{i,n}-x_{i,\vartheta_k(n)})}
{b_{k,n}^*-v_{k,n}^*}\nonumber\\
&\quad\;+\Scal{\sum_{i\in I}L_{ki}x_{i,\vartheta_k(n)}-b_{k,n}}
{v^*_{k,\vartheta_k(n)}-v_{k,n}^*}\nonumber\\
&\leq\bigg(\sum_{i\in I}\|L_{ki}\|
\sup_{j\in\NN}\big(\|b_{k,j}^*\|+\|v_{k,j}^*\|
\big)\bigg)\|x_{i,n}-x_{i,\vartheta_k(n)}\|\nonumber\\
&\quad\;+\sup_{j\in\NN}\bigg(\sum_{i\in I}\|L_{ki}\|\,
\|x_{i,\vartheta_k(j)}\|+\|b_{k,j}\|\bigg)
\|v^*_{k,n}-v_{k,\vartheta_k(n)}^*\|\nonumber\\
&\to 0.
\end{align}
In addition, \eqref{eBnd6Rfdr31g} and \eqref{eP9u66fbG07a} 
yield
\begin{align}
\label{eccx23gr24c}
(\forall k\in K)(\forall n\in\NN)\quad\widetilde{\psi}_{k,n}
&=\Scal{\sum_{i\in I}L_{ki}x_{i,\vartheta_k(n)}-b_{k,n}}
{b_{k,n}^*-v_{k,\vartheta_k(n)}^*}\nonumber\\
&=\mu_{k,\vartheta_k(n)}^{-1}
\bigg\|\Sum_{i\in I}L_{ki}x_{i,\vartheta_k(n)}-b_{k,n}\bigg\|^2
\nonumber\\
&=\mu_{k,{\vartheta_k(n)}}\|b_{k,n}^*-v_{k,{\vartheta_k(n)}}^*\|^2.
\end{align}
Consequently,
\begin{equation}
\label{eaHE5j-7h507u}
(\forall k\in K)(\forall n\in\NN)\quad
\widetilde{\psi}_{k,n}=\dfrac{\bigg\|\Sum_{i\in I}
L_{ki}x_{i,\vartheta_k(n)}-
b_{k,{\vartheta_k(n)}}\bigg\|^2}
{\mu_{k,{\vartheta_k(n)}}}\geq\varepsilon
\bigg\|\Sum_{i\in I}L_{ki}x_{i,\vartheta_k(n)}-
b_{k,\vartheta_k(n)}\bigg\|^2
\end{equation}
and
\begin{equation}
\label{eaHE5j-7h507t}
(\forall k\in K)(\forall n\in\NN)\quad
\widetilde{\psi}_{k,n}=\mu_{k,{\vartheta_k(n)}}
\|b_{k,n}^*-v_{k,{\vartheta_k(n)}}^*\|^2
\geq\varepsilon\|b_{k,n}^*-v_{k,{\vartheta_k(n)}}^*\|^2.
\end{equation}
It follows from \eqref{eaHE5j-7h507u} that
\begin{align}
\label{eaHE5j-7h505q}
&\hskip -6mm 
(\forall k\in K)(\forall n\in\NN)\quad
\bigg\|\Sum_{i\in I}L_{ki}x_{i,n}-b_{k,n}\bigg\|^2
\nonumber\\
&\hskip 32mm=
\bigg\|\Sum_{i\in I}L_{ki}(x_{i,n}-x_{i,\vartheta_k(n)})
+\Sum_{i\in I}L_{ki}x_{i,\vartheta_k(n)}-b_{k,n}\bigg\|^2
\nonumber\\
&\hskip 32mm\leq 2\bigg(\Sum_{i\in I}
\|L_{ki}\|^2\,\|x_{i,n}-x_{i,\vartheta_k(n)}\|^2+
\bigg\|\Sum_{i\in I}L_{ki}x_{i,\vartheta_k(n)}-b_{k,n}\bigg\|^2
\bigg)
\nonumber\\
&\hskip 32mm\leq 2\bigg(
\Sum_{i\in I}\|L_{ki}\|^2\,\|x_{i,n}-x_{i,\vartheta_k(n)}\|^2+
\varepsilon^{-1}(\widetilde{\psi}_{k,n}-\psi_{k,n})
+\varepsilon^{-1}\psi_{k,n}\bigg),
\end{align}
and from \eqref{eaHE5j-7h507t} that
\begin{align}
\label{eaHE5j-7h505p}
(\forall k\in K)(\forall n\in\NN)\quad
\|b^*_{k,n}-v^*_{k,n}\|^2
&\leq 2\big(\|b_{k,n}^*-v_{k,{\vartheta_k(n)}}^*\|^2+
\|v_{k,n}^*-v_{k,{\vartheta_k(n)}}^*\|^2\big)
\nonumber\\
&\leq 2\big(\varepsilon^{-1}(\widetilde{\psi}_{k,n}-\psi_{k,n})
+\varepsilon^{-1}\psi_{k,n}+
\|v_{k,n}^*-v_{k,{\vartheta_k(n)}}^*\|^2\big).
\end{align}
On the one hand, we derive from \eqref{eccx23gr28a},
\eqref{eaHE5j-7h505n}, and \eqref{eaHE5j-7h505p} that
\begin{align}
\label{eaHE5j-7h505z}
(\forall n\in\NN)\quad
\|x_n-a_n\|^2+\|b_n^*-v_n^*\|^2
&=\sum_{i\in I}\|x_{i,n}-a_{i,n}\|^2+
\sum_{k\in K}\|b^*_{k,n}-v^*_{k,n}\|^2\nonumber\\
&\leq2\sum_{i\in I}\|x_{i,n}-x_{i,\ell_i(n)}\|^2+
2\sum_{k\in K}\|v^*_{k,n}-v^*_{k,\vartheta_k(n)}\|^2
\nonumber\\
&\quad\;
+2\varepsilon^{-1}\sum_{i\in I}(\widetilde{\phi}_{i,n}-\phi_{i,n})
+2\varepsilon^{-1}\sum_{k\in K}(\widetilde{\psi}_{k,n}-\psi_{k,n})
\nonumber\\
&\quad\;
+2\varepsilon^{-1}\bigg(\sum_{i\in I}\phi_{i,n}+
\sum_{k\in K}\psi_{k,n}\bigg).
\end{align}
On the other hand, we derive from \eqref{eccx23gr28a},
\eqref{eaHE5j-7h505m}, and \eqref{eaHE5j-7h505q} that
\begin{align}
\label{eaHE5j-7h505y}
&\hskip -6mm (\forall n\in\NN)\quad
\|a_n^*+L^*v^*_n\|^2+\|Lx_n-b_n\|^2\nonumber\\
&\hskip 21mm=\sum_{i\in I}\bigg\|a_{i,n}^*+\sum_{k\in K}
L^*_{ki}v^*_{k,n}\bigg\|^2+\sum_{k\in K}\bigg\|\sum_{i\in I}
L_{ki}x_{i,n}-b_{k,n}\bigg\|^2\nonumber\\
&\hskip 21mm\leq 2\varepsilon^{-1}\Sum_{i\in I}
(\widetilde{\phi}_{i,n}-\phi_{i,n})+2\Sum_{i\in I}\Sum_{k\in K}
\|L_{ki}\|^2\,\|v^*_{k,\ell_i(n)}-v^*_{k,n}\|^2\nonumber\\
&\hskip 21mm\quad\;+2\Sum_{k\in K}\Sum_{i\in I}\|L_{ki}\|^2\,
\|x_{i,n}-x_{i,\vartheta_k(n)}\|^2+2\varepsilon^{-1}
\Sum_{k\in K}(\widetilde{\psi}_{k,n}-\psi_{k,n})\nonumber\\
&\hskip 21mm\quad\;
+2\varepsilon^{-1}\bigg(\sum_{i\in I}\phi_{i,n}+
\sum_{k\in K}\psi_{k,n}\bigg).
\end{align}
We deduce from 
\eqref{eaHE5j-7h504b}, \eqref{eccx23gr24b},
\eqref{eccx23gr28t}, \eqref{eccx23gr28a}, 
and \ref{aP9u66fbG07v} that 
\begin{equation}
\label{eccx23gr26u}
\varlimsup\bigg(\sum_{i\in I}{\phi}_{i,n}+
\sum_{k\in K}{\psi}_{k,n}\bigg)=
\varlimsup\big(\scal{x_n-a_n}{a_n^*+L^*v_n^*}
+\scal{Lx_n-b_n}{b_n^*-v_n^*}\big)\leq 0.
\end{equation}
Altogether, taking the limit superior in \eqref{eaHE5j-7h505z} 
and \eqref{eaHE5j-7h505y}, and using 
\eqref{eaHE5j-7h503h}, 
\eqref{eaHE5j-7h503b}, 
\eqref{eaHE5j-7h504c}, 
\eqref{eccx23gr24o}, 
\eqref{eaHE5j-7h503d}, 
\eqref{eaHE5j-7h503t}, 
and 
\eqref{eccx23gr26u}, 
we obtain
\begin{equation}
\label{eaHE5j-7h503w}
x_n-a_n\to 0,\quad
a_n^*+L^*v^*_n\to 0,\quad
Lx_n-b_n\to 0,
\quad\text{and}\quad b_n^*-v_n^*\to 0.
\end{equation}
Now set 
$x=(x_i)_{i\in I}$ and $v^*=(v_k^*)_{k\in K}$. 
Then \ref{aP9u66fbG07vi} and \eqref{eaHE5j-7h503w} yield
\begin{equation}
\label{eccx23gr28v}
a_{q_n}\weakly x,\quad b_{q_n}^*\weakly v^*,\quad 
a_{q_n}^*+L^*b_{q_n}^*\to 0,\quad\text{and}\quad 
La_{q_n}-b_{q_n}\to 0. 
\end{equation}
In turn, \ref{pP9u66fbG07i} and
Proposition~\ref{pgre76g12-12}\ref{pgre76g12-12iv} imply 
that $(x,v^*)\in\boldsymbol{Z}$. 
\end{proof}

\begin{remark}
\label{rFEW7d23}
In \eqref{eP9u66fbG07a}, the resolvents are assumed 
to be computed exactly to simplify the presentation. 
However, it is possible to allow for 
relative errors in these computations in the 
spirit of \cite[Algorithm~3]{Svai09}. More precisely,
we can replace the calculation
\begin{equation}
\label{e:xactAi}
(a_{i,n},a_{i,n}^*)=
\Big(J_{\gamma_{i,c_i(n)} A_i}\big(x_{i,c_i(n)}+\gamma_{i,c_i(n)}
(z^*_i-l^*_{i,n})\big),\gamma_{i,c_i(n)}^{-1}
(x_{i,c_i(n)}-a_{i,n})-l^*_{i,n}\Big)
\end{equation}
by any choice of $(a_{i,n},a_{i,n}^*)\in\HH_i^2$ such that
\begin{equation}
\label{e:rrordef}
(a_{i,n},z^*_i+a_{i,n}^*)\in\gra A_i\quad\text{and}\quad
a_{i,n}+\gamma_{i,c_i(n)}a_{i,n}^*=x_{i,c_i(n)}
-\gamma_{i,c_i(n)}l^*_{i,n}+e_{i,n},
\end{equation}
where the error $e_{i,n}$ satisfies 
\begin{equation}
\label{e:rrorcond}
\begin{cases}
\|e_{i,n}\|\leq\beta\\
\scal{e_{i,n}}{a^*_{i,n}+l^*_{i,n}}
\leq\sigma\gamma_{i,c_i(n)}
\|a^*_{i,n}+l^*_{i,n}\|^2\\
\scal{x_{i,c_i(n)}-a_{i,n}}{e_{i,n}}
\geq-\sigma 
\|x_{i,c_i(n)}-a_{i,n}\|^2
\end{cases}
\end{equation}
for some constants $\beta\in\RPP$ and $\sigma\in\zeroun$ that
are independent of $i$ and $n$. It follows from 
\cite[Proposition~23.21]{Livre1} 
that \eqref{e:rrordef} can also be written as
\begin{equation}
\label{eFEW7d25}
(a_{i,n},a_{i,n}^*)=
\Big(J_{\gamma_{i,c_i(n)} A_i}\big(x_{i,c_i(n)}+\gamma_{i,c_i(n)}
(z^*_i-l^*_{i,n})+e_{i,n}\big),\gamma_{i,c_i(n)}^{-1}
(x_{i,c_i(n)}-a_{i,n}+e_{i,n})-l^*_{i,n}\Big).
\end{equation}
It may easily be seen that the calculations \eqref{e:xactAi}  
satisfy \eqref{e:rrordef} with $e_{i,n}=0$, trivially fulfilling
\eqref{e:rrorcond}. In the setting of
\eqref{e:rrorcond}, \eqref{eaHE5j-7h503s} becomes
\begin{align}
\label{eFEW7d14a}
(\forall i\in I)(\forall n\in\NN)\quad
\widetilde{\phi}_{i,n}
&=\gamma_{i,\ell_i(n)}^{-1}\big(\|x_{i,\ell_i(n)}-a_{i,n}\|^2
+\scal{x_{i,\ell_i(n)}-a_{i,n}}{e_{i,n}}\big)\nonumber\\
&\geq\varepsilon(1-\sigma)\|x_{i,\ell_i(n)}-a_{i,n}\|^2
\end{align}
and \eqref{eaHE5j-7h503v} becomes
\begin{align}
\label{eFEW7d14b}
(\forall i\in I)(\forall n\in\NN)\quad
\widetilde{\phi}_{i,n}
&=\gamma_{i,\ell_i(n)}\bigg\|a_{i,n}^*+\Sum_{k\in K}L_{ki}^*
v^*_{k,\ell_i(n)}\bigg\|^2-\Scal{e_{i,n}}
{a_{i,n}^*+\Sum_{k\in K}L_{ki}^*v^*_{k,\ell_i(n)}}\nonumber\\
&\geq\varepsilon(1-\sigma)\bigg\|a_{i,n}^*+
\Sum_{k\in K}L_{ki}^*v^*_{k,\ell_i(n)}\bigg\|^2.
\end{align}
Likewise, we can replace the calculation
\begin{equation}
\label{e:xactBk}
(b_{k,n},b^*_{k,n})=\Big(r_k+J_{\mu_{k,d_k(n)}B_k}
\big(l_{k,n}+\mu_{k,d_k(n)}v_{k,d_k(n)}^*-r_k\big),
v_{k,d_k(n)}^*+\mu_{k,d_k(n)}^{-1}(l_{k,n}-b_{k,n})\Big)
\end{equation}
by any choice of $(b_{k,n},b_{k,n}^*)\in\GG_k^2$ such that
\begin{equation}
\label{errordef}
(b_{k,n}-r_k,b_{k,n}^*)\in\gra B_k\quad\text{and}\quad
b_{k,n}+\mu_{k,d_k(n)} b^*_{k,n}
=l_{k,n}+\mu_{k,d_k(n)}v_{k,d_k(n)}^*+f_{k,n},
\end{equation}
where the error $f_{k,n}$ satisfies 
\begin{equation}
\begin{cases}
\|f_{k,n}\|\leq\delta\\
\scal{l_{k,n}-b_{k,n}}{f_{k,n}}\geq-\zeta\|l_{k,n}-b_{k,n}\|^2\\
\scal{f_{k,n}}{b_{k,n}^*-v_{k,d_k(n)}^*}\leq\zeta\mu_{k,d_k(n)}
\|b^*_{k,n}-v^*_{k,d_k(n)}\|^2
\end{cases}
\label{errorcond}
\end{equation}
for some constants $\delta\in\RPP$ and $\zeta\in\zeroun$ that are
independent of $k$ and $n$. Altogether, the effect of such
approximate resolvent evaluations is to replace $\varepsilon^{-1}$
by $\varepsilon^{-1}(1-\sigma)^{-1}$ or
$\varepsilon^{-1}(1-\zeta)^{-1}$ in
\eqref{eaHE5j-7h505z}--\eqref{eaHE5j-7h505y}, with the
remainder of the proof of Proposition~\ref{pP9u66fbG07}
remaining unchanged.
\end{remark}

\subsection{A weakly convergent algorithm for finding a Kuhn-Tucker
point}

We propose a Fej\'er monotone primal-dual algorithm based on 
the results of Section~\ref{sec:fejer} to find a point in 
the Kuhn-Tucker set \eqref{eferw7er3h02-26k}.

\begin{algorithm}
\label{algo:3}
Consider the setting of Problem~\ref{prob:2}, let $\KKK$ 
be a closed vector subspace of 
$\bigoplus_{i\in I}\HH_i\oplus\bigoplus_{k\in K}\GG_k$ such that
$\boldsymbol{Z}\subset\KKK$, and suppose that 
Assumption~\ref{ass:validcontrol} is in force. Let 
$(\lambda_n)_{n\in\NN}\in
\left[\varepsilon,2-\varepsilon\right]^{\NN}$, let 
$((x_{i,0})_{i\in I},(v^*_{k,0})_{k\in K})\in\KKK$, and 
iterate
\begin{equation}
\label{eaHE5j-7h503a}
\begin{array}{l}
\text{for}\;n=0,1,\ldots\\
\left\lfloor
\begin{array}{l}
\begin{array}{l}
\text{for every}\;i\in I_n\\
\left\lfloor
\begin{array}{l}
l^*_{i,n}=\sum_{k\in K}L_{ki}^*v_{k,c_i(n)}^*\\
(a_{i,n},a_{i,n}^*)=
\Big(J_{\gamma_{i,c_i(n)} A_i}\big(x_{i,c_i(n)}+\gamma_{i,c_i(n)}
(z^*_i-l^*_{i,n})\big),\gamma_{i,c_i(n)}^{-1}
(x_{i,c_i(n)}-a_{i,n})-l^*_{i,n}\Big)\\
\end{array}
\right.\\[1mm]
\text{for every}\;i\in I\smallsetminus I_n\\
\left\lfloor
\begin{array}{l}
(a_{i,n},a_{i,n}^*)=(a_{i,n-1},a_{i,n-1}^*)\\
\end{array}
\right.\\[1mm]
\text{for every}\;k\in K_n\\
\left\lfloor
\begin{array}{l}
l_{k,n}=\sum_{i\in I}L_{ki}x_{i,d_k(n)}\\
(b_{k,n},b^*_{k,n})=\Big(r_k+J_{\mu_{k,d_k(n)}B_k}
\big(l_{k,n}+\mu_{k,d_k(n)}v_{k,d_k(n)}^*-r_k\big),
v_{k,d_k(n)}^*+\mu_{k,d_k(n)}^{-1}(l_{k,n}-b_{k,n})\Big)\\
\end{array}
\right.\\[1mm]
\text{for every}\;k\in K\smallsetminus K_n\\
\left\lfloor
\begin{array}{l}
(b_{k,n},b^*_{k,n})=(b_{k,n-1},b^*_{k,n-1})\\
\end{array}
\right.\\[1mm]
\big((t^*_{i,n})_{i\in I},(t_{k,n})_{k\in K}\big)=P_{\KKK}
\big((a^*_{i,n}+\sum_{k\in K}L_{ki}^*b^*_{k,n})_{i\in I},
(b_{k,n}-\sum_{i\in I}L_{ki}a_{i,n})_{k\in K}\big)\\
\tau_n=\sum_{i\in I}\|t_{i,n}^*\|^2+\sum_{k\in K}\|t_{k,n}\|^2\\
\text{if}\;\tau_n>0\\
\left\lfloor
\begin{array}{l}
\theta_n=\Frac{\lambda_n}{\tau_n}\,\text{\rm max} 
\braces{0,\sum_{i\in I}\big(\scal{x_{i,n}}{t^*_{i,n}}-
\scal{a_{i,n}}{a^*_{i,n}}\big)+\sum_{k\in K}
\big(\scal{t_{k,n}}{v^*_{k,n}} -\scal{b_{k,n}}{b^*_{k,n}}\big)}\\
\end{array}
\right.\\
\text{else~} \theta_n = 0 \\
\text{for every}\;i\in I\\
\left\lfloor
\begin{array}{l}
x_{i,n+1}=x_{i,n}-\theta_n t^*_{i,n}\\
\end{array}
\right.\\
\text{for every}\;k\in K\\
\left\lfloor
\begin{array}{l}
v^*_{k,n+1}=v^*_{k,n}-\theta_n t_{k,n}.
\end{array}
\right.\\
\end{array}
\end{array}
\right.\\[4mm]
\end{array}
\end{equation}
\end{algorithm}

\begin{remark}
\label{rferw7er3h07-12b}
When Problem~\ref{prob:1} has no special structure, on can take
$\KKK=\bigoplus_{i\in I}\HH_i\oplus\bigoplus_{k\in K}\GG_k$ 
in Algorithm~\ref{algo:3}. In other instances, it may be
advantageous computationally to use a suitable proper subspace 
$\KKK$. For instance, if $I=\{1\}$, $z^*_1=0$, and 
$A_1\colon\HH_1\to\HH_1$ is linear, then \eqref{eferw7er3h02-26k} 
reduces to
\begin{multline}
\label{eferw7er3h07-12a}
\boldsymbol{Z}=
\bigg\{\big(x_1,(v^*_k)_{k\in K}\big)\;\bigg |\;
x_1\in\HH_1,\;
A_1x_1+\sum_{k\in K}L_{k1}^*v_k^*=0,\:\;\text{and}\\
(\forall k\in K)\;\;v_k^*\in\GG_k\;\;\text{and}\;\;
L_{k1}x_1-r_k\in B_k^{-1}v_k^*\bigg\},
\end{multline}
and we can use
\begin{equation}
\label{eferw7er3h07-12b}
\KKK=\bigg\{\big(x_1,(v^*_k)_{k\in K}\big)\in
\HH_1\oplus\bigoplus_{k\in K}\GG_k\;\bigg |\;
A_1x_1+\sum_{k\in K}L_{k1}^*v_k^*=0\bigg\}.
\end{equation}
In effect, this approach was adopted in \cite{Svai09} in the
further special case in which $A_1=0$ and $(\forall k\in K)$
$\GG_k=\HH_1$, $L_{k1}=\Id$, and $r_k=0$.
\end{remark}

\begin{theorem}
\label{taHE5j-7h503}
Consider the setting of Problem~\ref{prob:2} and
Algorithm~\ref{algo:3}, suppose that
$\mathscr{P}\neq\emp$, and let
\begin{equation}
\label{eccx23gr28u}
(\forall n\in\NN)\quad
x_n=(x_{i,n})_{i\in I}\quad\text{and}\quad
v_n^*=(v_{k,n}^*)_{k\in K}.
\end{equation}
Then $(x_n)_{n\in\NN}$ converges weakly to a point 
$\overline{x}\in\mathscr{P}$, $(v^*_n)_{n\in\NN}$ converges weakly 
to a point $\overline{v}\in\mathscr{D}$, and 
$(\overline{x},\overline{v}^*)\in\boldsymbol{Z}$.
\end{theorem}
\begin{proof}
Define $\HH$, $\GG$, $L$, $A$, and $B$ as in 
\eqref{eccx23gr28s}, and
$(a_n)_{n\in\NN}$, $(a^*_n)_{n\in\NN}$, $(b_n)_{n\in\NN}$, and
$(b^*_n)_{n\in\NN}$ as in \eqref{eccx23gr28t}. Further, define
$(\forall n\in\NN)$ $t_n=(t_{k,n})_{k\in K}$ and 
$t_n^*=(t^*_{i,n})_{i\in I}$.
It follows from \eqref{eaHE5j-7h503a}, \eqref{eccx23gr28u},
\eqref{eccx23gr28m}, and 
Proposition~\ref{pP9u66fbG07}\ref{pP9u66fbG07i} that 
Algorithm~\ref{algo:3} is a special case of 
Algorithm~\ref{algo:1}. Hence, upon invoking
Proposition~\ref{paHE5j-7h502}, we can apply the results of
Proposition~\ref{p:2013-05-22} in this setting.
First, Proposition~\ref{p:2013-05-22}\ref{p:2013-05-22i} implies
that the boundedness assumption \ref{aP9u66fbG07i} in 
Proposition~\ref{pP9u66fbG07} is satisfied.
Second, in view of \eqref{eaHE5j-7h503a},
the sequence $(a_n,a^*_n)_{n\in\NN}$ and 
$(b_n,b^*_n)_{n\in\NN}$ are constructed according to 
assumption \ref{aP9u66fbG07iii} in 
Proposition~\ref{pP9u66fbG07}. We thus derive from 
Proposition~\ref{pP9u66fbG07}\ref{pP9u66fbG07ii} that
\begin{equation}
\label{eP9u66fbG09a}
(a_n)_{n\in\NN},\;\;
(a^*_n)_{n\in\NN},\;\;
(b_n)_{n\in\NN},\;\;\text{and}\;\;
(b^*_n)_{n\in\NN}\;\;\text{are bounded}. 
\end{equation}
Furthermore, the summability assumption \ref{aP9u66fbG07iv} in 
Proposition~\ref{pP9u66fbG07} is secured by
Proposition~\ref{p:2013-05-22}\ref{p:2013-05-22iii}, while
the limit superior assumption \ref{aP9u66fbG07v} in 
Proposition~\ref{pP9u66fbG07} holds by
Proposition~\ref{p:2013-05-22}\ref{p:2013-05-22iv}. We 
therefore use 
Proposition~\ref{pP9u66fbG07}\ref{pP9u66fbG07iii} 
to conclude by applying 
Proposition~\ref{p:2013-05-22}\ref{p:2013-05-22v}.
To this end, take $(x,v^*)\in\KKK$ and a strictly increasing 
sequence $(q_n)_{n\in\NN}$ in $\NN$ 
such that $x_{q_n}\weakly x$ and $v^*_{q_n}\weakly v^*$. 
Then Proposition~\ref{pP9u66fbG07}\ref{pP9u66fbG07iii} 
assert that $(x,v^*)\in\boldsymbol{Z}$. Thus, 
\eqref{e:2013-11-27a} is satisfied and the proof is complete.
\end{proof}

\begin{remark}
\label{rferw7er3h07-11a}
Theorem~\ref{taHE5j-7h503} 
subsumes \cite[Theorem~4.3]{Siop14}, which required the
following additional assumptions: the implementation is
synchronous, \emph{i.e.},
\begin{equation}
\label{eferw7er3h07-12t} 
(\forall n\in\NN)(\forall i\in I)(\forall k\in K) \quad
c_i(n)=d_k(n)=n,
\end{equation}
no proper subspace is used, \emph{i.e.},
\begin{equation}
\label{eferw7er3h07-12s} 
\KKK=\bigoplus_{i\in I}\HH_i\oplus\bigoplus_{k\in K}\GG_k, 
\end{equation}
the control is fully parallel, \emph{i.e.}, 
\begin{equation}
\label{eferw7er3h07-12u} 
(\forall n\in\NN)\quad I_n=I\quad\text{and}\quad K_n=K,
\end{equation}
and common proximal parameters are used in the sense that
\begin{equation}
\label{eferw7er3h07-12v} 
(\forall n\in\NN)(\forall i\in I)(\forall k\in K) \quad
\gamma_{i,n}=\gamma_n\quad\text{and}\quad\mu_{k,n}=\mu_n.
\end{equation}
Therefore, the proposed method also subsumes~\cite{Dong05} and
\cite[Proposition~3]{Svai08} (see also \cite{Baus09}), which are 
special cases of~\cite[Theorem~4.3]{Siop14}; see 
\cite[Examples~3.7 and 3.8]{Siop14} for details.
\end{remark}

\begin{remark}
\label{rferw7er3h07-11b}
Theorem~\ref{taHE5j-7h503} 
is closely related to \cite[Proposition~4.2]{Svai09} 
(see also \cite{Baus09}), which considers the special case
of Problem~\ref{prob:2} in which $I=\{1\}$, $z^*_1=0$, $A_1=0$, and 
$(\forall k\in K)$ $\GG_k=\HH_1$ and $L_{k1}=\Id$.  
If in this case one sets
\begin{equation}
\label{eferw7er3h07-12q}
\KKK=\Menge{\big(x_1,(v^*_k)_{k\in K}\big)\in
\HH_1^{p+1}}{\sum_{k\in K}v_k^*=0}
\end{equation}
in our algorithm, we recover the special case of the method of
\cite[Section 4]{Svai09} in which the parameter $\alpha_{ij}^k$
of~\cite[Proposition~4.2]{Svai09} is $1$ if $i=j$, and $0$
otherwise. Other settings of $\alpha_{ij}^k$ in \cite{Svai09} produce
algorithms that are not special cases of our scheme, but must
process the resolvent of every operator at every
iteration and remain fully synchronous as in 
\eqref{eferw7er3h07-12t} and \eqref{eferw7er3h07-12u}.
\end{remark}

\begin{remark}
\label{rFEW7d04a}
Recall that the resolvent of the subdifferential of a proper 
lower semicontinuous convex function $f\colon\HH\to\RX$ is
Moreau's proximity operator 
$(\Id+\partial f)^{-1}=\prox_f\colon x\mapsto
\text{argmin}_{y\in\HH}(f(y)+\|x-y\|^2/2)$ \cite{Livre1,Mor62b}. 
Now consider the setting of 
Problem~\ref{prob:3} and execute Algorithm~\ref{algo:3} with
$(\forall i\in I)$ $A_i=\partial f_i$ and 
$(\forall k\in K)$ $B_k=\partial g_k$. Then,
using the same arguments as in 
\cite[Proposition~5.4]{Siop13}, it follows from 
Theorem~\ref{taHE5j-7h503} that
$(x_n)_{n\in\NN}$ converges weakly to a solution
to \eqref{e:2012-10-23p} and that
$(v^*_n)_{n\in\NN}$ converges weakly to a solution
to \eqref{e:2012-10-23d}.
\end{remark}

\begin{remark}
\label{rFEW7d22}
The framework of \cite[Algorithm~3]{Svai09} for solving
\eqref{e:2009} allows for relative errors in the computation of 
the resolvents. Similar errors may be incorporated in 
Algorithm~\ref{algo:3} by adopting the approximate evaluation 
scheme of Remark~\ref{rFEW7d23} to select points in the
graphs of the monotone operators in \eqref{eaHE5j-7h503a}.
Since Proposition~\ref{pP9u66fbG07} remains valid with 
such approximate resolvent computations, so does 
Theorem~\ref{taHE5j-7h503}.
\end{remark}

\subsection{A best approximation result}

In this section we use the abstract Haugazeau-like algorithm of
Section~\ref{sec:haug} to devise a strongly convergent
asynchronous block-iterative method to construct the best
approximation to a reference point from the Kuhn-Tucker set
\eqref{eferw7er3h02-26k}.

\begin{algorithm}
\label{algo:4}
Consider the setting of Problem~\ref{prob:2}, let $\KKK$ be a 
closed vector subspace of 
$\bigoplus_{i\in I}\HH_i\oplus\bigoplus_{k\in K}\GG_k$ such that
$\boldsymbol{Z}\subset\KKK$, and suppose that 
Assumption~\ref{ass:validcontrol} is in force. 
Let $(\lambda_n)_{n\in\NN}\in\left[\varepsilon,1\right]^{\NN}$, let 
$((x_{i,0})_{i\in I},(v^*_{k,0})_{k\in K})\in\KKK$, and 
iterate
\pagebreak[1]
\begin{equation}
\label{eP9u66fbG03b}
\begin{array}{l}
\text{for}\;n=0,1,\ldots\\
\left\lfloor
\begin{array}{l}
\begin{array}{l}
\text{for every}\;i\in I_n\\
\left\lfloor
\begin{array}{l}
l^*_{i,n}=\sum_{k\in K}L_{ki}^*v_{k,c_i(n)}^*\\
(a_{i,n},a_{i,n}^*)=
\Big(J_{\gamma_{i,c_i(n)} A_i}\big(x_{i,c_i(n)}+\gamma_{i,c_i(n)}
(z^*_i-l^*_{i,n})\big),\gamma_{i,c_i(n)}^{-1}
(x_{i,c_i(n)}-a_{i,n})-l^*_{i,n}\Big)\\[1mm]
\end{array}
\right.\\[1mm]
\text{for every}\;i\in I\smallsetminus I_n\\
\left\lfloor
\begin{array}{l}
(a_{i,n},a_{i,n}^*)=(a_{i,n-1},a_{i,n-1}^*)\\
\end{array}
\right.\\[1mm]
\text{for every}\;k\in K_n\\
\left\lfloor
\begin{array}{l}
l_{k,n}=\sum_{i\in I}L_{ki}x_{i,d_k(n)}\\
(b_{k,n},b^*_{k,n})=\Big(r_k+J_{\mu_{k,d_k(n)}B_k}
\big(l_{k,n}+\mu_{k,d_k(n)}v_{k,d_k(n)}^*-r_k\big),
v_{k,d_k(n)}^*+\mu_{k,d_k(n)}^{-1}(l_{k,n}-b_{k,n})\Big)\\
\end{array}
\right.\\[1mm]
\text{for every}\;k\in K\smallsetminus K_n\\
\left\lfloor
\begin{array}{l}
(b_{k,n},b^*_{k,n})=(b_{k,n-1},b^*_{k,n-1})\\
\end{array}
\right.\\[1mm]
\big((t^*_{i,n})_{i\in I},(t_{k,n})_{k\in K}\big)=P_{\KKK}
\big((a^*_{i,n}+\sum_{k\in K}L_{ki}^*b^*_{k,n})_{i\in I},
(b_{k,n}-\sum_{i\in I}L_{ki}a_{i,n})_{k\in K}\big)\\
\tau_n=\sum_{i\in I}\|t_{i,n}^*\|^2+\sum_{k\in K}\|t_{k,n}\|^2\\
\text{if}\;\tau_n>0\\
\left\lfloor
\begin{array}{l}
\theta_n=\Frac{\lambda_n}{\tau_n}\,\text{\rm max} 
\braces{0,\sum_{i\in I}\big(\scal{x_{i,n}}{t^*_{i,n}}-
\scal{a_{i,n}}{a^*_{i,n}}\big)+\sum_{k\in K}
\big(\scal{t_{k,n}}{v^*_{k,n}} -\scal{b_{k,n}}{b^*_{k,n}}\big)}\\
\end{array}
\right.\\
\text{else}\;\theta_n = 0 \\
\text{for every}\;i\in I\\
\left\lfloor
\begin{array}{l}
x_{i,n+1/2}=x_{i,n}-\theta_n t^*_{i,n}\\
\end{array}
\right.\\
\text{for every}\;k\in K\\
\left\lfloor
\begin{array}{l}
v^*_{k,n+1/2}=v^*_{k,n}-\theta_n t_{k,n}\\
\end{array}
\right.\\
\end{array}\\
\chi_n=\sum_{i\in I}\scal{x_{i,0}-x_{i,n}}{x_{i,n}-x_{i,n+1/2}}
+\sum_{k\in K}\scal{v_{k,0}^*-v_{k,n}^*}{v_{k,n}^*-v_{k,n+1/2}^*}\\
\mu_n=\sum_{i\in I}\|x_{i,0}-x_{i,n}\|^2+\sum_{k\in K}
\|v_{k,0}^*-v_{k,n}^*\|^2\\
\nu_n=\sum_{i\in I}\|x_{i,n}-x_{i,n+1/2}\|^2+
\sum_{k\in K}\|v_{k,n}^*-v_{k,n+1/2}^*\|^2\\
\rho_n=\mu_n\nu_n-\chi_n^2\\
\text{if}\;\rho_n=0\;\text{and}\;\chi_n\geq 0\\
\left\lfloor
\begin{array}{l}
\text{for every}\;i\in I\\
\left\lfloor
\begin{array}{l}
x_{i,n+1}=x_{i,n+1/2}\\
\end{array}
\right.\\
\text{for every}\;k\in K\\
\left\lfloor
\begin{array}{l}
v^*_{k,n+1}=v_{k,n+1/2}^*\\
\end{array}
\right.\\
\end{array}
\right.\\
\text{if}\;\rho_n>0\;\text{and}\;\chi_n\nu_n\geq\rho_n\\
\left\lfloor
\begin{array}{l}
\text{for every}\;i\in I\\
\left\lfloor
\begin{array}{l}
x_{i,n+1}=x_{i,0}+(1+\chi_n/\nu_n)(x_{i,n+1/2}-x_{i,n})\\
\end{array}
\right.\\
\text{for every}\;k\in K\\
\left\lfloor
\begin{array}{l}
v^*_{k,n+1}=v_{k,0}^*+(1+\chi_n/\nu_n)(v_{k,n+1/2}^*-v_{k,n}^*)
\end{array}
\right.\\
\end{array}
\right.\\
\text{if}\;\rho_n>0\;\text{and}\;\chi_n\nu_n<\rho_n\\
\left\lfloor
\begin{array}{l}
\text{for every}\;i\in I\\
\left\lfloor
\begin{array}{l}
x_{i,n+1}=x_{i,n}+(\nu_n/\rho_n)\big(\chi_n(x_{i,0}-x_{i,n})
+\mu_n(x_{i,n+1/2}-x_{i,n})\big)\\
\end{array}
\right.\\
\text{for every}\;k\in K\\
\left\lfloor
\begin{array}{l}
v^*_{k,n+1}=v_{k,n}^*+(\nu_n/\rho_n)
\big(\chi_n(v_{k,0}^*-v_{k,n}^*)
+\mu_n(v_{k,n+1/2}^*-v_{k,n}^*)\big).
\end{array}
\right.\\
\end{array}
\right.\\
\end{array}
\right.\\[4mm]
\end{array}
\end{equation}
\end{algorithm}
\pagebreak[1]

\begin{theorem}
\label{tP9u66fbG04}
Consider the setting of Problem~\ref{prob:2} and
Algorithm~\ref{algo:4}, and suppose that
$\mathscr{P}\neq\emp$. Define
\begin{equation}
\label{eccx23gr28z}
(\forall n\in\NN)\quad
x_n=(x_{i,n})_{i\in I}\quad\text{and}\quad
v_n^*=(v_{k,n}^*)_{k\in K}
\end{equation}
and set $(\overline{x},\overline{v}^*)=
P_{\boldsymbol{Z}}(x_0,v^*_0)$.
Then $(x_n)_{n\in\NN}$ converges strongly to 
$\overline{x}\in\mathscr{P}$ and $(v_n^*)_{n\in\NN}$ converges 
strongly to $\overline{v}^*\in\mathscr{D}$.
\end{theorem}
\begin{proof}
Define $\HH$, $\GG$, $L$, $A$, and $B$ as in 
\eqref{eccx23gr28s}, 
$(a_n)_{n\in\NN}$, $(a^*_n)_{n\in\NN}$, $(b_n)_{n\in\NN}$, and
$(b^*_n)_{n\in\NN}$ as in \eqref{eccx23gr28t}, and set
$(\forall n\in\NN)$ $t_n=(t_{k,n})_{k\in K}$ and 
$t_n^*=(t^*_{i,n})_{i\in I}$.
In view of \eqref{eP9u66fbG03b}, \eqref{eccx23gr28z},
\eqref{eccx23gr28m}, and 
Proposition~\ref{pP9u66fbG07}\ref{pP9u66fbG07i},
Algorithm~\ref{algo:4} is an instance of 
Algorithm~\ref{algo:2}. Hence, upon invoking
Proposition~\ref{paHE5j-7h502}, we can apply the results of
Proposition~\ref{pP9u66fbG03} in this setting.
First, Proposition~\ref{pP9u66fbG03}\ref{pP9u66fbG03i} 
implies that assumption \ref{aP9u66fbG07i} in 
Proposition~\ref{pP9u66fbG07} is satisfied.
Second, in view of 
\eqref{eP9u66fbG03b}, assumption \ref{aP9u66fbG07iii} in 
Proposition~\ref{pP9u66fbG07} is satisfied as well.
Thus, Proposition~\ref{pP9u66fbG07}\ref{pP9u66fbG07ii} 
asserts that the sequences
$(a_{n})_{n\in\NN}$, $(a^*_{n})_{n\in\NN}$, $(b_{n})_{n\in\NN}$,
and $(b^*_{n})_{n\in\NN}$ are bounded. 
Third, assumption \ref{aP9u66fbG07iv} in 
Proposition~\ref{pP9u66fbG07} is secured by
Proposition~\ref{pP9u66fbG03}\ref{pP9u66fbG03ii}.
Finally, assumption \ref{aP9u66fbG07v} in 
Proposition~\ref{pP9u66fbG07} holds by
Proposition~\ref{pP9u66fbG03}\ref{pP9u66fbG03iv}. We 
therefore use 
Proposition~\ref{pP9u66fbG07}\ref{pP9u66fbG07iii} to
conclude by invoking
Proposition~\ref{pP9u66fbG03}\ref{pP9u66fbG03v}. 
Take $(x,v^*)\in\KKK$ 
and a strictly increasing sequence $(q_n)_{n\in\NN}$ in $\NN$ 
such that $x_{q_n}\weakly x$ and $v^*_{q_n}\weakly v^*$. 
Then it follows from 
Proposition~\ref{pP9u66fbG07}\ref{pP9u66fbG07iii} that 
$(x,v^*)\in\boldsymbol{Z}$, which completes the proof.
\end{proof}

\begin{remark}
\label{rZgt70R04b}
As in Remark~\ref{rFEW7d04a}, 
consider the setting of Problem~\ref{prob:3} and execute
Algorithm~\ref{algo:4} with
$(\forall i\in I)$ $A_i=\partial f_i$ and 
$(\forall k\in K)$ $B_k=\partial g_k$. 
Then Theorem~\ref{tP9u66fbG04} asserts that
$(x_n)_{n\in\NN}$ converges strongly to a solution
$\overline{x}$ to \eqref{e:2012-10-23p} and that
$(v^*_n)_{n\in\NN}$ converges strongly to a solution
$\overline{v}^*$ to \eqref{e:2012-10-23d} such that 
$(\overline{x},\overline{v}^*)$ is the projection of 
$(x_0,v_0^*)$ onto the corresponding Kuhn-Tucker set
\eqref{eferw7er3h02-26k}.
\end{remark}

\begin{remark}
\label{rferw7er3h07-12}
Theorem~\ref{tP9u66fbG04} improves upon 
\cite[Proposition~4.2]{Nfao15}, which addresses the special 
case in which the algorithm is synchronous and the restrictions
\eqref{eferw7er3h07-12t}--\eqref{eferw7er3h07-12v} are imposed. The 
latter was applied in the context of Remark~\ref{rZgt70R04b}
to domain decomposition methods in \cite{Nume15}; 
Theorem~\ref{tP9u66fbG04} provides a new
range of ways to revisit such applications using
asynchronous block-iterative calculations.
\end{remark}

\begin{remark}
\label{rFEW7d21}
By an argument similar to that of Remark~\ref{rFEW7d22},
Theorem~\ref{tP9u66fbG04} remains valid if the resolvent
computations in~\eqref{eP9u66fbG03b} are replaced by
approximate evaluations meeting the conditions in 
Remark~\ref{rFEW7d23}.
\end{remark}

\end{document}